\documentclass[11pt,reqno]{amsart}

\newtheorem{thm}{Theorem}[section]
\newtheorem*{thm*}{Theorem}
\newtheorem{lem}[thm]{Lemma}
\newtheorem*{lem*}{Lemma}

\newtheorem{claim}[thm]{Claim}

\newtheorem{prop}[thm]{Proposition}

\theoremstyle{definition}

\renewcommand{\thecase}{}
\newtheorem*{case*}{Case}

\newtheorem{defn}[thm]{Definition}
\newtheorem*{defn*}{Definition}

\newtheorem*{exmp*}{Example}

\newtheorem{rmk}[thm]{Remark}
\newtheorem*{rmk*}{Remark}
\newtheorem{step}{Step}\renewcommand{\thestep}{}

\theoremstyle{remark}

\makeatletter
\def\alphenumi{
  \def\theenumi{\alph{enumi}}
  \def\p@enumi{\theenumi}
  \def\labelenumi{(\@alph\c@enumi)}}
\makeatother

\makeatletter
\def\thecase{\@arabic\c@case}
\makeatother

\makeatletter
\def\thestep{\@arabic\c@step}
\makeatother

\newcount\hh
\newcount\mm
\mm=\time
\hh=\time
\divide\hh by 60
\divide\mm by 60
\multiply\mm by 60
\mm=-\mm
\advance\mm by \time
\def\hhmm{\number\hh:\ifnum\mm<10{}0\fi\number\mm}

\let\oldmarginpar\marginpar
\renewcommand\marginpar[1]{\-\oldmarginpar[\raggedleft\footnotesize #1]%
{\raggedright\footnotesize #1}}

\newcommand\dotprod{\hbox{$\cdot$}}

\newcommand\HS{{\mathbb{R}^{n+1}_+}}

\newcommand\NN{\mathbb{N}}

\newcommand\RR{\mathbb{R}}

\newcommand\cB{{\mathcal{B}}}
\newcommand\cC{{\mathcal{C}}}

\newcommand\cH{{\mathcal{H}}}

\newcommand\eps{\varepsilon}

\newcommand\Om{\Omega}

\newcommand\divg{\operatorname{div}}

\newcommand\loc{\operatorname{loc}}

\newcommand\supp{\operatorname{supp}}

\numberwithin{equation}{section}

\usepackage{amssymb}
\usepackage{hyperref}
\usepackage{mathrsfs}
\usepackage{slashed}
\usepackage[usenames]{color}
\usepackage{url}
\usepackage{verbatim}

\hypersetup{pdftitle={Regularity of the free boundary for the obstacle problem for the fractional Laplacian with drift}}

\hypersetup{pdfauthor={Nicola Garofalo, Arshak Petrosyan, Camelia A. Pop, and Mariana Smit Vega Garcia}}

\begin{document}

\title[The obstacle problem for the fractional Laplacian with drift]{Regularity of the free boundary for the obstacle problem for the fractional Laplacian with drift}

\author[N. Garofalo]{Nicola Garofalo}
\address[NG]{Dipartimento d'Ingegneria Civile e Ambientale (DICEA), Universit\'a di Padova, Via Trieste 63, 35131 Padova, Italy}
\email{nicola.garofalo@unipd.it}

\author[A. Petrosyan]{Arshak Petrosyan}
\address[AP]{Department of Mathematics, Purdue University, West Lafayette, IN 47907}
\email{arshak@math.purdue.edu}

\author[C. A. Pop]{Camelia A. Pop}
\address[CP]{School of Mathematics, University of Minnesota, Minneapolis, MN 55455}
\email{capop@umn.edu}

\author[M. Smit Vega Garcia]{Mariana Smit Vega Garcia}
\address[MSVG]{Fakult\"at f\"ur Mathematik, Universit\"at Duisburg-Essen, 45117 Essen, Germany}
\email{mariana.vega-smit@uni-due.de}

\thanks{(NG) is partially supported by a grant ``Progetti d'Ateneo 2014" from the University of Padova. (CP) is partially supported by research funds provided by the School of Mathematics at University of Minnesota. (MSVG) is partially supported by DFG, ``Projekt Singularit{\"a}ten ElektroHydroDynamischer Gleichungen".}

\begin{abstract}
We establish the $C^{1+\gamma}$-H\"older regularity of the regular free boundary in the stationary obstacle problem defined by the fractional Laplace operator with drift in the subcritical regime. Our method of the proof consists in proving a new monotonicity formula and an epiperimetric inequality. Both tools generalizes the original ideas of G. Weiss in \cite{Weiss_1999} for the classical obstacle problem to the framework of fractional powers of the Laplace operator with drift. Our study continues the earlier research \cite{Petrosyan_Pop}, where two of us established the optimal interior regularity of solutions.
\end{abstract}

%


\maketitle

\tableofcontents

\section{Introduction}
\label{sec:Intro}

In this paper we continue the study initiated in \cite{Petrosyan_Pop} of the obstacle problem 
\begin{equation}
\label{eq:Obstacle_problem}
\min\{L \widehat u(x), \widehat u(x) - \widehat\varphi(x)\}=0,\quad\forall\, x\in \RR^n,
\end{equation}
where we have denoted by $L$ the fractional Laplacian operator with drift defined by
\begin{equation}
\label{eq:Operator}
L\psi(x) :=\left(-\Delta\right)^s \psi(x) + b(x)\dotprod\nabla \psi (x)+c(x) \psi(x),\quad\forall\, \psi \in C^2_0(\RR^n).
\end{equation}
For $0<s<1$ the action of the fractional Laplacian $(-\Delta)^s$  on functions $\psi\in C^2_0(\RR^n)$ is given by the singular integral,
\begin{equation}
\label{eq:Fractional_laplacian}
(-\Delta)^s \psi(x) = c_{n,s} \operatorname{p.v.} \int_{\RR^n}\frac{\psi(x)-\psi(y)}{|x-y|^{n+2s}}\, dy,
\end{equation}
which is understood in the sense of the principal value. The constant $c_{n,s}$ in \eqref{eq:Fractional_laplacian} is positive and depends only on the dimension $n\in \NN$ and on the parameter $s$. The range $(0,1)$ of the parameter $s$ is particularly interesting because in this case the fractional Laplacian operator is the infinitesimal generator of the symmetric $2s$-stable process \cite[Example 3.3.8]{Applebaum}.

In the subcritical regime, that is, when $s\in (1/2,1)$, in \cite[Theorem~1.1]{Petrosyan_Pop} two of us established the existence and the optimal regularity $\widehat u \in C^{1+s}(\RR^n)$ of the solution to the problem \eqref{eq:Obstacle_problem} under the assumptions that $b\in C^s(\mathbb R^n;\mathbb R^n)$, $c\in C^s(\mathbb R^n)$, with $c\ge 0$, and the obstacle $\widehat\varphi\in C^{3s}(\mathbb R^n)\cap C_0(\mathbb R^n)$, and satisfies $(L\widehat\varphi)^+\in L^\infty(\mathbb R^n)$. Furthermore, if $b$ is Lipschitz continuous and $c\ge c_0>0$, the solution is unique. For the definition of the H\"older spaces $C^r(\mathbb R^n)$ we refer the reader to \S\ref{sec:fs} below. 

The assumption $s\in (1/2,1)$ plays a crucial role in  \cite{Petrosyan_Pop} since it allows to treat the drift term in the definition \eqref{eq:Operator} of $L$ as a lower-order term. This assertion is made precise in \S\ref{sec:No_drift}, where we also explain the technical difficulties caused by the lower-order terms $b$ and $c$ in the definition \eqref{eq:Obstacle_problem} of the operator $L$.

In the present article we continue the study of the obstacle problem \eqref{eq:Obstacle_problem}. In our main result,  Theorem~\ref{thm:Regularity_free_boundary_with_drift} below, we establish the $C^{1+\gamma}$-H\"older continuity of the free boundary in the neighborhood of any \emph{regular} free boundary point.

\subsection{Reduction to an obstacle problem for the fractional Laplacian without drift}
\label{sec:No_drift}
In \cite[\S2.3]{Petrosyan_Pop} it was proved that the study of the obstacle problem with drift \eqref{eq:Obstacle_problem} can be reduced to one \emph{without} drift in the following way. Given a solution $\widehat u \in C^{1+s}(\RR^n)$ to \eqref{eq:Obstacle_problem} we construct a function $w\in C^{3s}(\RR^n)$ as a solution to the linear equation,
$$
(-\Delta)^s w = b(x)\dotprod\nabla\widehat u+c(x)\widehat u.
$$
Applying the second part of 
\cite[Proposition~2.8]{Silvestre_2007} with $\alpha = \sigma = s$ (note that since $1/2<s<1$ we have $\alpha + 2 \sigma = 3s>1$), and using the fact that the right-hand side in the latter equation is in $C^s(\RR^n)$, we have that the function $w$ belongs to $C^{3s}(\RR^n)$. We now define
\begin{align*}
u:=\widehat u-w,\quad\text{and}\quad \varphi:=\widehat\varphi-w.
\end{align*}
Since $s>1/2$ we have $3s > 1+ s$ and thus $C^{3s}(\mathbb R^n)$ is continuously embedded into $C^{1+s}(\mathbb R^n)$, see \S\ref{sec:fs}, and thus $u\in C^{1+s}(\RR^n)$. Such $u$ is a solution to the obstacle problem defined by the fractional Laplacian operator \emph{without} drift,
\begin{equation}
\label{eq:Obstacle_problem_without_drift}
\min\{\left(-\Delta\right)^s u(x), u(x)-\varphi(x)\}=0,\quad\forall\, x\in \RR^n.
\end{equation}
We remark that because of the preceding reduction procedure to an obstacle problem without drift, the obstacle function $\varphi$ can be assumed at most to belong to the H\"older space $C^{3s}(\RR^n)$, even when the obstacle function $\widehat\varphi$, in problem \eqref{eq:Obstacle_problem}, is assumed to be a smooth function. This is the main technical difference in the study of the fractional Laplacian operator with drift, and the one without drift.

\subsection{Main result}\label{sec:Main_result}

To state our main result concerning the regularity of the free boundary  we henceforth indicate with
$$
\widehat \Gamma(\widehat u) := \partial\{\widehat u = \widehat\varphi\}.
$$
the set of free boundary points corresponding to the obsta\-cle problem for the fractional Laplacian with drift \eqref{eq:Obstacle_problem}. Likewise, the notation
\begin{align*}
\label{eq:Free_boundary_set}
\Gamma(u) &:= \partial\{u=\varphi\}
\end{align*}
will indicate the free boundary for problem \eqref{eq:Obstacle_problem_without_drift}. We notice that the reduction procedure from an obstacle problem with drift to one without drift described in \S\ref{sec:No_drift} above implies that $\widehat\Gamma(\widehat u)=\Gamma(u)$. 
Henceforth, we denote by $\Gamma_{1+s}(u)$ the subset of $\Gamma(u)$ composed of \emph{regular free boundary points} for the problem \eqref{eq:Obstacle_problem_without_drift} according to Definition~\ref{defn:Regular_points} below. 

We can now define the set of regular free boundary points for problem \eqref{eq:Obstacle_problem}.

\begin{defn}\label{defn:Regular_points_with_drift}
We say that a free boundary point $x_0\in\widehat\Gamma(\widehat u)$ is \emph{regular} for problem \eqref{eq:Obstacle_problem} if $x_0$ is a regular free boundary point for problem \eqref{eq:Obstacle_problem_without_drift}, i.e., $x_0\in \Gamma_{1+s}(u)$. If we denote by $\widehat\Gamma_{1+s}(\widehat u)$ the set of regular free boundary points for problem \eqref{eq:Obstacle_problem}, then according to our definition we have $\widehat\Gamma_{1+s}(\widehat u) = \Gamma_{1+s}(u)$.
\end{defn}

The following two theorems are the main results of this paper.

\begin{thm}[$C^{1+\gamma}$ regularity of the regular free boundary for problem \eqref{eq:Obstacle_problem_without_drift}]
\label{thm:Regularity_free_boundary}
Let $s\in (1/2,1)$, and let $u\in C^{1+s}(\RR^n)$ be a solution to the obstacle problem \eqref{eq:Obstacle_problem_without_drift}, where the obstacle function $\varphi\in C^{3s}(\RR^n)$. Let $x_0\in\Gamma_{1+s}(u)$. Then, there are positive constants, $\gamma=\gamma(\kappa,n,s)\in (0,1)$ and $\eta$, such that $B'_{\eta}(x_0)\cap\Gamma(u)\subseteq \Gamma_{1+s}(u)$, and there is a function, $g\in C^{1+\gamma}(\RR^{n-1})$, such that, after a possible rotation of the system of coordinates in $\RR^n$, one has
\begin{equation}
\label{eq:Regularity_free_boundary}
B'_{\eta}(x_0)\cap\Gamma(u) =
B'_{\eta}(x_0)\cap\{x=(x',x_n)\in\RR^{n-1}\times\RR \mid x_n\leq g(x')\}.
\end{equation}
\end{thm}

To put Theorem~\ref{thm:Regularity_free_boundary} in the proper
historical perspective we recall that when the obstacle is assumed to
belong to $C^{2,1}(\RR^n)$, the $C^{1+\gamma}$-H\"older continuity of
the regular free boundary for the obstacle problem
\eqref{eq:Obstacle_problem_without_drift} was obtained by Caffarelli,
Salsa, and Silvestre, see [Theorem~7.7] in \cite{Caffarelli_Salsa_Silvestre_2008}. In this paper we improve on this result by establishing the regularity of the free boundary under the weaker condition  that $\varphi\in C^{3s}(\RR^n)$, which is crucial in our proof of Theorem~\ref{thm:Regularity_free_boundary_with_drift}. This limitation in the regularity of the obstacle function makes the method of the proof of \cite[Theorem~7.7]{Caffarelli_Salsa_Silvestre_2008} inapplicable to our framework. 
Our approach to Theorem~\ref{thm:Regularity_free_boundary} is based on adaptation of the Weiss monotonicity formula (\cite[Theorem~3.1]{Weiss_1998}, \cite[Theorem~2]{Weiss_1999}), and on a suitable epiperimetric inequality (\cite[Theorem~1]{Weiss_1999}). Similar ideas have been recently used in \cite{Garofalo_Petrosyan_SmitVegaGarcia} to establish the $C^{1+\gamma}$-H\"older continuity of the regular free boundary in the Signorini problem with variable coefficients (see \cite[Theorem~1.2]{Garofalo_Petrosyan_SmitVegaGarcia}).

\begin{thm}[$C^{1+\gamma}$ regularity of the regular free boundary for problem \eqref{eq:Obstacle_problem}]
\label{thm:Regularity_free_boundary_with_drift}
Let $s \in (1/2,1)$, and assume that $b \in C^{s}(\RR^n;\RR^n)$ and $c \in C^{s}(\RR^n)$. Let $\widehat u \in C^{1+s}(\RR^n)$ be a solution to the obstacle problem \eqref{eq:Obstacle_problem} for the fractional Laplacian with drift,  where the obstacle  $\widehat\varphi \in C^{3s}(\RR^n)$. Let $x_0\in\widehat\Gamma_{1+s}(\widehat u)$. Then, there exist positive constants $\gamma=\gamma(\kappa,n,s)\in (0,1)$ and $\eta$, such that $B'_{\eta}(x_0)\cap\widehat\Gamma(\widehat u)\subseteq \widehat\Gamma_{1+s}(\widehat u)$, and there is a function $g\in C^{1+\gamma}(\RR^{n-1})$ such that, after a possible rotation of the system of coordinates in $\RR^n$, one has
\begin{equation}
\label{eq:Regularity_free_boundary_with_drift}
B'_{\eta}(x_0)\cap\widehat\Gamma(\widehat u) =
B'_{\eta}(x_0)\cap\{x=(x',x_n)\in\RR^{n-1}\times\RR \mid x_n\leq g(x')\}.
\end{equation}
\end{thm}

\subsection{Outline of the article}
\label{sec:Outline}
In \S\ref{sec:Regular_points_Almgren} we recall the Almgren-type monotonicity formula, established in \cite[Propositions~2.12 and 2.13]{Petrosyan_Pop}, with the aid of which we define the concept of \emph{regular} free boundary points for problem \eqref{eq:Obstacle_problem_without_drift}. In \S\ref{sec:Weiss_monotonicity_homogeneous_res} we prove a Weiss-type monotonicity formula adapted to our framework, and we introduce the sequence of homogeneous rescalings at regular free boundary points together with some of the main properties which are extensively used in the sequel. In \S\ref{S:epi} we establish in Theorem~\ref{T:epi} a generalization of the epiperimetric inequality first obtained by Weiss in \cite[Theorem~1]{Weiss_1999} in the analysis of the classical obstacle problem. In \S\ref{sec:Regularity_free_boundary} we finally prove our main results, Theorems~\ref{thm:Regularity_free_boundary} and \ref{thm:Regularity_free_boundary_with_drift}. In Appendix \S\ref{sec:Auxiliary_results} we prove various auxiliary results that we use throughout the article.

\subsection{Notations and conventions}
\label{sec:Notation}
With $\RR_+:=(0,\infty)$, we denote by $\HS$  the upper half-space
$\mathbb R^n \times \mathbb R_+$. If $v,w\in\RR^n$, we let $v\dotprod
w$ indicate their scalar product. For $x_0\in\RR^{n+1}$ and $r>0$, let
$B_r(x_0)$ be the Euclidean ball in $\RR^{n+1}$ of radius $r$ centered
at $x_0$, and for $x_0\in\RR^n$ and $r>0$ we indicate with $B'_r(x_0)$
the Euclidean ball in $\RR^n$ of radius $r$ centered at $x_0$. We
denote by $B^+_r(x_0)$ the half-ball,
$B_r(x_0)\cap\left(\RR^{n}\times\RR_+\right)$. For brevity, when
$x_0=0$, we write $B_r$, $B'_r$, and $B^+_r$ instead of $B_r(0)$,
$B'_r(0)$, and $B^+_r(0)$, respectively.

For a set $S\subseteq\RR^n$, we denote its complement by $S^c:=\RR^n\setminus S$, and we let $\operatorname{int}(S)$ denote its topological interior.

For any real numbers, $a$ and $b$, we denote $a\wedge b:=\min\{a,b\}$.

\subsection{Function spaces}
\label{sec:fs}
 
In what follows we will need the H\"older spaces $C^{k+\alpha}(\Omega)$, where $\Omega\subset \mathbb R^n$ is an open set. We recall that for any $k\in \mathbb N_0 = \mathbb N \cup\{0\}$ the space  $C^k(\Omega)$ is the Banach space of the functions $u\in C^k_{\loc}(\Omega)$ such that the norm
\[
|u|_{k;\Omega} = \sum_{j=0}^k [u]_{j;\Omega} < \infty,
\]
where 
\[
[f]_{0;\Omega} = \sup_\Omega|f|,\quad [f]_{j;\Omega} = \sup_\Omega
\max_{|\alpha| = j} |D^\alpha f|.
\]
Notice that $|f|_{0;\Omega} = [f]_{0;\Omega}$. For $0<\delta<1$ we say that $u$ is $\delta$-H\"older continuous in $\Omega$ if the seminorm
\[
[u]_{\delta;\Omega} = \sup_{x, y\in \Omega, x\neq y} \frac{|u(x) - u(y)|}{|x-y|^\delta} <\infty.
\]
When $\delta = 1$ we say that $u$ is Lipschitz continuous in $\Om$. We let
\[
[u]_{k+\delta;\Omega} = \max_{|\alpha| = k}[D^\alpha u]_{\delta;\Omega}.
\]
For $0<\delta<1$ and $k\in \mathbb N \cup \{0\}$ we define $C^{k+\delta}(\Omega)$ as the Banach space of functions in $C^k(\Omega)$ such that the norm 
\[
|u|_{k+\delta;\Omega} = |u|_{k;\Omega} + [u]_{k+\delta;\Omega} <\infty.
\]
When $\Omega = \mathbb R^n$ we simply write $C^{k+\delta}$ instead of $C^{k+\delta}(\mathbb R^n)$. Let us note explicitly that when $k=0$ the space $C^{\delta}(\Omega)$ is defined as the set of functions in $C(\Omega)$ which are $\delta$-H\"older continuous in $\Omega$ and such that
\[
|u|_{\delta;\Omega} = |u|_{0;\Omega} + [u]_{\delta;\Omega} <\infty.
\]
We will often make use of the simple observation that if $u, v\in C^{\delta}(\Omega)$, then $u v\in C^{\delta}(\Omega)$ as well. Also, we note that if $r\ge s\ge 0$, then $C^r(\Omega) \subset C^s(\Omega)$, with the inclusion being continuous.
This can be seen as follows. Let $u\in C^r(\Omega)$ and $x, y\in \Omega$. Suppose first that $|x-y|\le 1$. Then,
\[
\frac{|u(x) - u(y)|}{|x-y|^s} = \frac{|u(x) - u(y)|}{|x-y|^r} |x-y|^{r-s}\le [u]_{r;\Omega}.
\]
This gives $[u]_{s;\Omega} \le [u]_{r;\Omega}$. If instead $|x-y|\ge 1$, then
\[
|u(x) - u(y)| \le 2 |u|_{0;\Omega} \le 2 |u|_{0;\Omega} |x-y|^s.
\]
This gives $[u]_{s;\Omega} \le 2 |u|_{0;\Omega}\le 2 |u|_{r;\Omega}$.

One should pay attention to the fact that, although the spaces
$C^{r}(\Omega)$ are defined for every $r\ge 0$, when $r\in \mathbb N$
it is not true that $C^{r}(\Omega) = C^{(r-1)+1}(\Omega)$ according to
our definition of the spaces $C^{k+\delta}(\Omega)$; i.e.,\
$C^{r}(\Omega)$ is not the space of functions having $r-1$ Lipschitz
continuous derivatives in $\Omega$.

Finally, we will need the weighted H\"older spaces $C^{1+\alpha}_a(\bar \Omega)$, where $\alpha \in (0,1)$, $\Omega \subseteq \HS$ is an open set, and we recall that $a=1-2s$. A function $u \in C^1(\Omega)$ is said to belong to $C^{1+\alpha}_a(\bar\Omega)$ if
\begin{equation}
\label{eq:Schauder_space}
\|u\|_{C^{1+\alpha}_a(\bar\Omega)}:=\|u\|_{C^{\alpha}(\bar\Omega)} + \|u_{x_i}\|_{C^{\alpha}(\bar\Omega)} + \||y|^a\partial_y u\|_{C^{\alpha}(\bar\Omega)} <\infty.
\end{equation}

\section{Regular free boundary points and Almgren rescalings}
\label{sec:Regular_points_Almgren}
We divide this section into two parts. In \S\ref{sec:Regular_points} we review the Almgren-type monotonicity formula introduced in \cite{Petrosyan_Pop}  which we use to define the notion of regular free boundary points. In \S\ref{sec:Almgren_rescalings}, we recall the definition of the Almgren rescalings and we establish some of their properties, which play a fundamental role in the study of the regularity of the free boundary in a neighborhood of free boundary points.

\subsection{Regular free boundary points}
\label{sec:Regular_points}
In this section we give the definition of regular free boundary points, and we establish some of their properties which will be used in the sequel. 

Let $a:=1-2s$. We consider the operator $L_a$ defined, for all $v\in C^2(\HS)$, by
\begin{equation}
\label{eq:L_a}
L_a v(x,y) = \divg(|y|^{a}\nabla v)(x,y),\quad(x,y)\in \HS.
\end{equation}
The relation between the degenerate-elliptic operator $L_a$ and the fractional Laplacian operator, $(-\Delta)^s$, is investigated in \cite[\S3]{Caffarelli_Silvestre_2007}, where it is established that $L_a$-harmonic functions, $u$, satisfy
\begin{equation}
\label{eq:Dirichlet_to_Neumann_map}
\lim_{y\downarrow 0} y^a u_y(x,y) = -(-\Delta)^s u(x,0),
\end{equation}
where identity \eqref{eq:Dirichlet_to_Neumann_map} holds up to multiplication by a constant factor (see \cite[Formula (3.1)]{Caffarelli_Silvestre_2007}).
In other words, the fractional Laplacian operator, $(-\Delta)^s$, is a Dirichlet-to-Neumann map for the elliptic operator $L_a$. For a probabilistic interpretation of the relationship between the fractional Laplacian operator $(-\Delta)^s$, and the degenerate-elliptic operator $L_a$, see  \cite{Molchanov_Ostrovskii_1969}, where the authors establish that the $2s$-symmetric stable process, with infinitesimal generator $(-\Delta)^s$, is a Brownian motion subordinated with the inverse local time of a Bessel process, with infinitesimal generator $L_a$.

We fix a point $x_0\in\Gamma(u)$. Following \cite[Definition (2.41)]{Petrosyan_Pop}, we introduce the height function,
\begin{equation}
\label{eq:Height_function}
v_{x_0}(x,y) := u(x,y) - \varphi(x,y) - \frac{1}{2s}(-\Delta)^s\varphi(x_0)|y|^{1-a},
\end{equation}
where the functions $u(x,y)$ and $\varphi(x,y)$ are the $L_a$-harmonic extensions of $u(x)$ and $\varphi(x)$ from $\RR^n$ to $\HS$. When $x_0=0$, we write for brevity $v(x,y)$ instead of $v_{0}(x,y)$. From \cite[Equations (2.43), (2.44), (2.46), and (2.47)]{Petrosyan_Pop} we recall that the height function $v_{x_0}(x,y)$ satisfies the conditions:
\begin{align}
\label{eq:Properties_v_1}
L_a v_{x_0} =0 &\quad\text{in }\RR^n\times(\RR\setminus\{0\}),\\
\label{eq:Properties_v_2}
v_{x_0} \geq 0&\quad\text{on }\RR^n\times\{0\},\\
\label{eq:Upper_bound_L_a}
L_a v_{x_0}(x,y) \leq h_{x_0}(x)\cH^n|_{\{y=0\}}
&\quad\text{on } \RR^{n+1},\\
\label{eq:Equality_L_a}
L_a v_{x_0}(x,y) = h_{x_0}(x) \cH^n|_{\{y=0\}}
&\quad\text{on } \RR^{n+1}\setminus(\{y=0\}\cap\{v_{x_0}=0\}),
\end{align}
where the source function $h_{x_0}$ is defined by
$$
h_{x_0}(x):=2\left((-\Delta)^s \varphi(x)-(-\Delta)^s
  \varphi(x_0)\right),\quad x\in \RR^n.
$$
From the construction \eqref{eq:Height_function} of the height function $v_{x_0}$ and from \cite[Theorem~1.1]{Petrosyan_Pop}, it follows that $h_{x_0}$ belongs to $C^s(\RR^n)$, and there is a positive constant, $C$, such that 
\begin{equation}
\label{eq:Growth_v_h_on_R_n}
|v_{x_0}(x, 0)| \leq C|x|^{1+s},
\quad\text{and}\quad
|h_{x_0}(x, 0)| \leq C|x|^s,
\quad x \in \RR^n.
\end{equation}
We recall the Almgren-type monotonicity formula associated to the function $v_{x_0}(x,y)$ that two of us established in \cite[Proposition~2.12]{Petrosyan_Pop}. For this purpose, we first need to introduce suitable weighted Sobolev spaces. Let $U\subseteq\RR^{n+1}$ be a Borel set. We say that a function $w$ belongs to the weighted Sobolev space $H^1(U,|y|^a)$, if $w$ and $Dw$ are function in $L^2_{\loc}(U, |y|^a)$ and
$$
\int_{U}\left(|w|^2+|\nabla w|^2\right)|y|^a <\infty.
$$
From \cite[\S2.4]{Caffarelli_Silvestre_2007} it follows that the auxiliary function $v_{x_0}(x_0+\cdot)$ belongs to the spaces $C(\RR^{n+1})$ and $ H^1(B_r,|y|^a)$, for all $r>0$. In particular, the following quantities are well-defined:
\begin{align}
\label{eq:F}
F_{x_0}(r)&:=\int_{\partial B_r} |v_{x_0}(x_0+\cdot)|^2 |y|^a,\\
\label{eq:d_r}
d_{x_0,r}&:=\left(\frac{1}{r^{n+a}} F_{x_0}(r)\right)^{1/2},\\
\label{eq:Phi}
\Phi^p_{x_0}(r) &:= r\frac{d}{dr} \log\max\{F_{x_0}(r), r^{n+a+2(1+p)}\},
\end{align}
where $r>0$ and $p>0$. The functions $F_{x_0}(r)$ and $\Phi^p_{x_0}(r)$ are the analogues of the functions $F_u(r)$ and $\Phi_u(r)$ given by \cite[Definitions (3.1) and (3.2)]{Caffarelli_Salsa_Silvestre_2008}, but adapted to our framework. We can now state the following result which combines Propositions~2.12 and 2.13 from \cite{Petrosyan_Pop}.

\begin{prop}[Almgren-type monotonicity formula]
\label{prop:Monotonicity_formula}
Let $s\in (1/2, 1)$, $\alpha \in (1/2, s)$, and $x_0$ be a free boundary point. Then, for all $p\in [s, \alpha+s-1/2)$ there exist positive constants, 
$C = C(\|u\|_{C^{1+\alpha}(\RR^n)})$, $\gamma = 2(\alpha+s-p)-1$, and $r_0=r_0(\alpha, p, s, \|u\|_{C^{1+\alpha}(\RR^n)})\in (0,1)$, such that the function
\begin{equation}
\label{eq:Monotonicity_formula}
(0,r_0)\ni r\mapsto e^{Cr^{\gamma}} \Phi^p_{x_0}(r),
\end{equation}
is nondecreasing. Moreover, if
\begin{equation}
\label{eq:Fraction_d_r_r_power_finite}
\liminf_{r\downarrow 0}\frac{d_{x_0,r}}{r^{1+p}} <\infty,
\end{equation}
then
\begin{equation}
\label{eq:Phi_at_0_p}
\Phi^p_{x_0}(0+) = n+a+2(1+p),
\end{equation}
and if
\begin{equation}
\label{eq:Fraction_d_r_r_power_infty}
\liminf_{r\downarrow 0}\frac{d_{x_0,r}}{r^{1+p}} =\infty,
\end{equation}
then
\begin{equation}
\label{eq:Phi_at_0}
\Phi^p_{x_0}(0+) \geq n+a+2(1+s).
\end{equation}
\end{prop}

We also have a straightforward consequence of the proof of \cite[Proposition~2.13]{Petrosyan_Pop}.

\begin{prop}[Property of the Almgren-type monotonicity formula]
\label{prop:Property_monotonicity_formula}
Let $s\in (1/2, 1)$, $\alpha \in (1/2, s)$, and $x_0\in\Gamma(u)$. Then for all $p\in [s, \alpha+s-1/2)$, we have that either one of the following three possibilities occurs:
\begin{equation*}
\Phi^p_{x_0}(0+) = n+a+2(1+s),\quad \Phi^p_{x_0}(0+) = n+a+2(1+p),\quad\text{or}\quad \Phi^p_{x_0}(0+) \geq n+a+4.
\end{equation*}
\end{prop}

\begin{proof}
In the case \eqref{eq:Fraction_d_r_r_power_finite}, we have from Proposition~\ref{prop:Monotonicity_formula} that identity \eqref{eq:Phi_at_0_p} holds. It only remains to analyze the case when condition \eqref{eq:Fraction_d_r_r_power_infty} holds. Following the proof of \cite[Proposition~2.13]{Petrosyan_Pop}, we see that either \eqref{eq:Phi_at_0_p} holds, or $\Phi^p_{x_0}(0+)=\Phi_{v_{x_0}}(0+)$, where $\Phi_{v_{x_0}}(r)$ is the Almgren monotonicity formula defined in \cite[Formula (3.2)]{Caffarelli_Salsa_Silvestre_2008}. From \cite[Lemma~6.1]{Caffarelli_Salsa_Silvestre_2008}, it follows that $\Phi_{v_{x_0}}(0+) = n+a+2(1+s)$, or $\Phi_{v_{x_0}}(0+) \geq n+a+4$. Thus the conclusion of Proposition~\ref{prop:Property_monotonicity_formula} holds.
\end{proof}

We can now give the definition of regular free boundary points for problem \eqref{eq:Obstacle_problem_without_drift}.

\begin{defn}\label{defn:Regular_points}
We say that a free boundary point $x_0\in\Gamma(u)$ is \emph{regular} for problem \eqref{eq:Obstacle_problem_without_drift} if
\begin{equation}
\label{eq:Regular_points}
\Phi^p_{x_0}(0+) = n+a+2(1+s),\quad\forall\, p\in (s, 2s-1/2).
\end{equation}
The set of regular free boundary points will be denoted by $\Gamma_{1+s}(u)$.
\end{defn}

We have the following.

\begin{lem}[Property of regular free boundary points]
\label{lem:Property_regular_points}
Let $x_0\in\Gamma(u)$. If there exists $q\in (s,2s-1/2)$ such that $\Phi^q_{x_0}(0+) = n+a+2(1+s)$, then $x_0$ is a regular free boundary point.
\end{lem}

\begin{proof}
Because $\Phi^q_{x_0}(0+) = n+a+2(1+s)$, it follows from Proposition~\ref{prop:Monotonicity_formula} that property \eqref{eq:Fraction_d_r_r_power_infty} holds with $p=q$, and so using definitions \eqref{eq:F} and \eqref{eq:d_r}, we have that $F_{x_0}(r) > r^{n+a+2(1+q)}$, for $r$ small enough. This implies that
$$
\Phi^q_{x_0}(r) = r\frac{F'_{x_0}(r)}{F_{x_0}(r)}.
$$
Making use of the monotonicity of the function $r\mapsto e^{Cr^{\gamma}} \Phi^q_{x_0}(r)$, and the fact that $\Phi^q_{x_0}(0+) = n+a+2(1+s)$, we obtain that for all $\eps>0$, there is a positive constant $r_{\eps}$ such that
$$
r\frac{F'_{x_0}(r)}{F_{x_0}(r)} < n+a+2(1+s)+\eps,\quad\forall\, r\in (0,r_{\eps}).
$$
Integrating in $r$, we obtain that we can find a positive constant, $C_{\eps}$, such that
$$
F_{x_0}(r) \geq C_{\eps} r^{n+a+2(1+s)+\eps},\quad\forall\, r\in (0,r_{\eps}).
$$
Given $p \in (s, 2s-1/2)$, we choose $\eps>0$ small enough such that $2s+\eps<2p$, which gives 
$$
\max\{F_{x_0}(r), r^{n+a+2(1+p)}\} = F_{x_0}(r),\quad\forall\, r\in (0,r_{\eps}).
$$
From definition \eqref{eq:Phi} of the function $\Phi^p_{x_0}(r)$ we obtain that
$$
\Phi^p_{x_0}(r) = r\frac{F'_{x_0}(r)}{F_{x_0}(r)} = \Phi^q_{x_0}(r),\quad\forall\, r\in (0,r_{\eps}),
$$
and  thus we conclude that $\Phi^p_{x_0}(0+) = n+a+2(1+s)$ for all $p \in (s, 2s-1/2)$. It follows that $x_0$ is a regular free boundary point.
\end{proof}

We now have the following analogue of \cite[Lemma~3.3]{Garofalo_Petrosyan_SmitVegaGarcia} which shows that the set of regular free boundary points is open in the relative topology of the free boundary.

\begin{lem}
\label{lem:Open_set_regular_points}
Let $x_0\in\Gamma_{1+s}(u)$. Then, there is a positive constant, $\eta = \eta(x_0)$, such that
$$
B'_{\eta}(x_0) \cap \Gamma(u) \subseteq \Gamma_{1+s}(u).
$$
Moreover, for all $p\in (s, 2s-1/2)$, the convergence
\begin{equation}
\label{eq:Uniform_convergence_Phi_x_r}
\Phi^p_x(r) \rightarrow n+a+2(1+s),\quad\text{as } r \downarrow 0
\end{equation}
is uniform, for all $x \in B'_{\eta}(x_0) \cap \Gamma(u)$.
\end{lem}

\begin{proof}
Our method of the proof follows that of \cite[Lemma~3.3]{Garofalo_Petrosyan_SmitVegaGarcia}, but contains small variations because the Dirichlet-to-Neumann map $(-\Delta)^{1/2} = \partial_{\nu}$ in \cite{Garofalo_Petrosyan_SmitVegaGarcia} is replaced by the fractional Laplacian operator $(-\Delta)^s$ in our problem. Let $p \in (s, 2s-1/2)$, and choose a constant $\eps \in (0, (p\wedge 1-s)/2)$. Our goal is to first show that there are positive constants, $\eta = \eta(\eps, x_0)$ and $\rho = \rho(\eps, x_0)$, such that
\begin{equation}
\label{eq:Phi_p_rho}
\Phi^p_x(\rho) < n+a+2(1+s)+\eps,\quad\forall\, x\in B'_{\eta}(x_0)\cap\Gamma(u).
\end{equation}
Because $x_0\in\Gamma_{1+s}(u)$, it follows from Definition~\ref{defn:Regular_points} that
$$
\Phi^p_{x_0}(r) = r\frac{F'_{x_0}(r)}{F_{x_0}(r)} \rightarrow n+a+2(1+s),\quad\text{as } r \downarrow 0,
$$
and so there is a positive constant, $\rho = \rho(\eps, x_0) < r_0/2$, where $r_0$ is given by Proposition~\ref{prop:Monotonicity_formula}, such that
\begin{equation}
\label{eq:Continuity_in_r}
r\frac{F'_{x_0}(r)}{F_{x_0}(r)} < n+a+2(1+s) + \frac{\eps}{3},\quad\forall\, r \in (0, 2\rho).
\end{equation}
Using \cite[Theorem~1.1]{Petrosyan_Pop}, it follows that the function
$$
\Gamma(u) \ni x \mapsto \rho\frac{F'_x(\rho)}{F_x(\rho)}
$$
is continuous. Combined with inequality \eqref{eq:Continuity_in_r}, this implies the existence of a positive constant $\eta = \eta(\eps, x_0)$ such that
\begin{equation}
\label{eq:Continuity_in_x}
\rho\frac{F'_x(\rho)}{F_x(\rho)} < n+a+2(1+s) + \frac{2\eps}{3},\quad\forall\, x \in B'_{\eta}(x_0) \cap\Gamma(u).
\end{equation}
For $x \in B'_{\eta}(x_0) \cap\Gamma(u)$ fixed, because the function
$$
(0,\infty) \ni r \mapsto r\frac{F'_x(r)}{F_x(r)}
$$
is continuous, we obtain from inequality \eqref{eq:Continuity_in_x} that there is a positive constant $\delta = \delta(\eps, x) <\rho$ such that 
\begin{equation}
\label{eq:Almost_ineq_Phi_p}
r\frac{F'_x(r)}{F_x(r)} < n+a+2(1+s) + \eps,\quad\forall\, r\in (\rho-\delta, \rho+\delta).
\end{equation}
Integrating in $r$ the previous inequality gives us that there is a positive constant, $c$, such that
$$
F_{x}(r) > cr^{n+a+2(1+s)+\eps},\quad\forall\, r\in (\rho-\delta, \rho+\delta).
$$
Because we have chosen $\eps \in (0, (p\wedge 1-s)/2)$ we see that
$$
\max\{F_x(r), r^{n+a+2(1+p)}\} = F_x(r),\quad\forall\, r\in (\rho-\delta, \rho+\delta).
$$
Using definition \eqref{eq:Phi} of the function $\Phi^p_x(r)$, together with \eqref{eq:Almost_ineq_Phi_p}, it follows that inequality \eqref{eq:Phi_p_rho} holds. Without loss of generality, we may assume that the positive constant $\rho$ is chosen small enough so that
$$
e^{Cr^{\gamma}} < \frac{n+a+2(1+s) + 2\eps}{n+a+2(1+s) + \eps},\quad\forall\, r\in [0,2\rho).
$$
Combined with \eqref{eq:Phi_p_rho} this implies that
$$
e^{C\rho^{\gamma}} \Phi^p_{x}(\rho) < n+a+2(1+s) + 2\eps.
$$
Recalling that we have chosen $\eps \in (0, (p\wedge 1-s)/2)$, the preceding inequality implies
$$
e^{C\rho^{\gamma}} \Phi^p_{x}(\rho) < n+a+2(1+p\wedge 1).
$$
Applying Proposition~\ref{prop:Property_monotonicity_formula}, we obtain that $\Phi^p_x(0+) = n+a+2(1+s)$, which implies that $x$ belongs to $\Gamma_{1+s}(u)$, whenever $x \in B'_{\eta}(x_0)\cap\Gamma(u)$.

The uniform convergence \eqref{eq:Uniform_convergence_Phi_x_r} in $x \in B'_{\eta}(x_0) \cap \Gamma(u)$ is a consequence of Dini's Theorem 
since, with $x_0$ replaced by a fixed $x \in B'_{\eta}(x_0)\cap\Gamma(u)$, the function \eqref{eq:Monotonicity_formula}  is nondecreasing, while for fixed $r \in (0,\rho]$, the function $\Gamma(u)\ni x\mapsto e^{Cr^{\gamma}} \Phi^p_x(r)$ is continuous.
This concludes the proof.
\end{proof}

\subsection{Almgren rescalings}
\label{sec:Almgren_rescalings}
We now discuss properties of the sequence of Almgren-type resca\-lings $\{\tilde v_{x_0,r}\}_{r >0}$ of the function $v$. We recall their definition from \cite[Identity (2.54)]{Petrosyan_Pop}:
\begin{equation}
\label{eq:Rescaling}
\tilde v_{x_0,r}(x,y):=\frac{v(x_0 +rx,ry)}{d_{x_0,r}},\quad(x,y)\in\RR^n\times\RR,
\end{equation}
where $d_{x_0,r}$ is defined in \eqref{eq:d_r}, and $x_0\in\Gamma(u)$. When $x_0=0$ we write for brevity $\tilde v_r$ instead of $\tilde v_{x_0,r}$. We first need to introduce the set $\cH_{1+s}$ consisting of homogeneous functions on $\RR^{n+1}$ of degree $1+s$ of the form:
\begin{equation}
\label{eq:H_1_plus_s}
\cH_{1+s}=\left\{ a \left(x\dotprod e+\sqrt{(x \dotprod e)^2+y^2}\right)^s
\left(x\dotprod e -s\sqrt{(x \dotprod e)^2+y^2}\right)\mid e\in\partial B'_1,\, a\ge 0 \right\}.
\end{equation}
We have the following properties of the sequence of rescalings around a regular free boundary point:
\begin{lem}
There exists $c>0$ such that for all $x_0\in\Gamma_{1+s}(u)$ and all $p\in (s,2s-1/2)$ one has:
\begin{enumerate}
\item[(i)] Property \eqref{eq:Fraction_d_r_r_power_infty} holds.
\item[(ii)] There exists $r_0=r_0(p, x_0)>0$ such that
\begin{equation}
\label{eq:Phi_at_regular_points}
\Phi^p_{x_0}(r) = r\frac{F'_{x_0}(r)}{F_{x_0}(r)},\quad\forall\, r\in (0,r_0).
\end{equation}
\item[(iii)]
The sequence of rescalings $\{\tilde v_{x_0,r}\}_{r>0}$ contains a subsequence that converges strongly in $H^1(B^+_{1/8}, |y|^a)$ to a homogeneous function $\tilde v_{x_0} \in \cH_{1+s}$; i.e.,\ there exists $e \in\partial B'_1$ such that
\begin{equation}
\label{eq:Limit_Alm_res}
\tilde v_{x_0} = c \left(x\dotprod e+\sqrt{(x \dotprod e)^2+y^2}\right)^s
\left(x\dotprod e - s\sqrt{(x \dotprod e)^2+y^2}\right).
\end{equation}
Moreover, the function $\tilde v_0(x,y)$ satisfies the system of conditions:
\begin{equation}
\label{eq:Eq_v_0}
\begin{aligned}
\tilde v_{x_0} \geq 0&\quad\text{on } \RR^n\times\{0\},\\
\tilde v_{x_0}(x,y) = \tilde v_{x_0}(x,-y),&\quad\forall\,  (x,y)\in \RR^n\times\RR_+,\\
L_a \tilde v_{x_0} =0 &\quad\text{on }\RR^{n+1}\setminus \left( \RR^n\times\{0\}\cap\{\tilde v_{x_0}=0\}\right),\\
L_a \tilde v_{x_0} \leq 0&\quad\text{on } \RR^{n+1},
\end{aligned}
\end{equation}
\end{enumerate}
\end{lem}

\begin{proof}
Properties (i) and (ii) are a straightforward consequence of Proposition~\ref{prop:Monotonicity_formula}. We now give the proof of property (iii). The existence of a subsequence of the sequence of rescalings that converges strongly in $H^1(B^+_{1/8}, |y|^a)$ to a function $\tilde v_{x_0}$, and satisfies the conditions \eqref{eq:Eq_v_0}, follows from the proof of \cite[Proposition~2.13]{Petrosyan_Pop}. From \cite[Identities (2.95) and (2.97)]{Petrosyan_Pop} we observe that 
$$
\frac{r\int_{B_r}|\nabla \tilde v_{x_0}|^2|y|^a}{\int_{\partial B_r}|\tilde v_{x_0}|^2|y|^a} =1+s,
$$
for all $r>0$ small enough. It follows from \cite[Theorem~6.1]{Caffarelli_Silvestre_2007} that $\tilde v_{x_0}$ is a homogeneous function of degree $1+s$. From \cite[Lemma~4.1]{Caffarelli_Salsa_Silvestre_2008} we obtain that $\tilde v_{x_0}(\cdot, 0)$ is a semiconvex function, and because it is a homogeneous, we have that $\tilde  v_{x_0}(\cdot,0)$ is convex. We can now apply \cite[Proposition~5.5]{Caffarelli_Salsa_Silvestre_2008} to conclude that there is a  real constant $c$, and a direction $e\in \partial B'_1$, such that $\partial\{\tilde  v_{x_0}=0\}\cap \RR^n\times\{0\}$ is a half-space and the representation formula \eqref{eq:Limit_Alm_res} holds. The fact that the positive constant $c$ is independent of the choice of the free boundary point $x_0$ follows from the fact that $\|\tilde  v_{x_0}\|_{L^2(\partial B_1, \|y\|^a)} = 1$, which is clear from the definition \eqref{eq:Rescaling} of the sequence of rescalings. 
\end{proof}

We next state an analogue of
\cite[Lemma~3.4]{Garofalo_Petrosyan_SmitVegaGarcia}, which shows a
locally uniform convergence of the Almgren rescalings $\tilde v_{x,r}$ to the homogeneous
functions in $\cH_{1+s}$ in the weighed $C^{1+\alpha}_a$-norm, as defined
in \eqref{eq:Schauder_space}.

\begin{lem}[Convergence to homogeneous functions]
\label{lem:Convergence_homogeneous_functions}
Let $x_0\in\Gamma_{1+s}(u)$. There exists $\alpha\in (0,1)$ such that for all $\eps>0$ one can find $r_0=r_0(x_0)>0$ and $\eta=\eta(x_0)>0$ for which
\begin{equation}
\label{eq:Convergence_homogeneous_functions}
\inf_{v \in \cH_{1+s}} \|\tilde  v_{x,r} - v\|_{C^{1+\alpha}_a(\bar B^+_{1/8})} < \eps,
\end{equation}
for all $r\in (0,r_0)$ and all $x\in B'_{\eta}(x_0)\cap\Gamma_{1+s}(u)$.
\end{lem}

Before proving Lemma~\ref{lem:Convergence_homogeneous_functions} we establish the following uniform a priori local Schauder estimates.

\begin{lem}\label{lem:Uniform_Schauder_estimates}
Let $x_0\in\Gamma_{1+s}(u)$. Then, there exist constants $\alpha\in (0,1)$, $C>0$, $\eta>0$ and $r_0>0$, such that for all $r\in (0, r_0)$ and every $x \in B'_{\eta}(x_0)\cap \Gamma_{1+s}(u)$
\begin{equation}
\label{eq:Uniform_Schauder_estimates}
\|\tilde v_{x,r}\|_{C^{1+\alpha}_a(\bar B^+_{1/8})} \leq C.
\end{equation}
\end{lem}

\begin{proof}
Let $\eta=\eta(x_0)>0$ be chosen as in the statement of Lemma~\ref{lem:Open_set_regular_points}. In \cite[Lemma~2.17]{Petrosyan_Pop} an estimate similar to \eqref{eq:Uniform_Schauder_estimates} was obtained, but with the constant $r_0 = r_0(x)>0$ depending on the free boundary point $x \in B'_{\eta}(x_0)\cap \Gamma_{1+s}(u)$. From the proof of \cite[Lemma~2.17]{Petrosyan_Pop} we can trace the dependence of the constant $r_0(x)$ on the validity of \cite[Inequality (2.62)]{Petrosyan_Pop}. That is, for all $x \in B'_{\eta}(x_0)\cap \Gamma_{1+s}(u)$ there exists $r_0(x)>0$ such that
\begin{equation}
\label{eq:F_r_lower_bound}
F_x(r) \geq r^{n+a+2(1+p)},\quad\forall\, r\in (0,r_0(x)),
\end{equation}
where $p\in (2, 2s-1/2)$ is any fixed constant. We now show that we can choose uniformly the positive constant $r_0(x)$, depending only on $x_0$. From property \eqref{eq:Uniform_convergence_Phi_x_r}, given $\eps\in (0,p-s)$ there exists $r_0=r_0(x_0)>0$ such that $\Phi^p_x(r) <n+a+2(1+s)+\eps$, for all $r\in (0, r_0)$ and every $x \in B'_{\eta}(x_0)\cap \Gamma_{1+s}(u)$. Using the definition \eqref{eq:Phi} of the function $\Phi^p_x(r)$, this implies that
$$
r\frac{d}{dr} \log\max\{F_x(r), r^{n+a+2(1+p)}\} < n+a+2(1+s)+\eps,
$$
for all $r\in (0,r_0)$, and every $x \in B'_{\eta}(x_0)\cap \Gamma_{1+s}(u)$. Integrating in $r$ the latter inequality we obtain the existence of $C=C(n,s,\|u\|_{C(\RR^n)},x_0)>0$ such that
$$
F_x(r) \geq C r^{n+a+2(1+s+\eps)},
\quad\forall\, r\in (0,r_0),
\quad\forall\, x \in B'_{\eta}(x_0)\cap \Gamma_{1+s}(u).
$$
This immediately implies inequality \eqref{eq:F_r_lower_bound}. We may now conclude with the aid of \cite[Lemma~2.17]{Petrosyan_Pop} that the uniform local Schauder estimate \eqref{eq:Uniform_Schauder_estimates} holds.
\end{proof}

We can now present the

\begin{proof}[Proof of Lemma~\ref{lem:Convergence_homogeneous_functions}]
The proof can be obtained using the same argument as the one in the proof of \cite[Lemma~3.4]{Garofalo_Petrosyan_SmitVegaGarcia}, with the observation that we must replace the $C^{1+1/2}(\bar B^+_1)$ uniform Schauder estimates of the sequence of rescalings with the ones obtained in Lemma~\ref{lem:Uniform_Schauder_estimates}. Furthermore, we need to replace the application of \cite[Theorem~2.4 and Lemma~2.5]{Garofalo_Petrosyan_SmitVegaGarcia} with that of Proposition~\ref{prop:Monotonicity_formula}.
\end{proof}

\section{A Weiss-type monotonicity formula and homogeneous rescalings at regular free boundary points}
\label{sec:Weiss_monotonicity_homogeneous_res}

In this section, we introduce in \S\ref{sec:Weiss_monotonicity} a Weiss-type functional and establish its monotonicity property. We then discuss in \S\ref{sec:Homogeneous_res} the homogeneous rescalings and some of their properties which are used extensively in the sequel.

\subsection{Weiss-type  monotonicity formula}
\label{sec:Weiss_monotonicity}
Let $x_0\in\Gamma(u)$ and $v_{x_0}(x,y)$ be the height function defined in \eqref{eq:Height_function}. We let
\begin{equation}
\label{eq:I}
I_{x_0}(r) := \int_{B_r(x_0)} |\nabla v_{x_0}|^2|y|^a + \int_{B'_r(x_0)} v_{x_0} h_{x_0},\quad\forall\, r\in \RR_+.
\end{equation}

Following \cite[p. 25]{Weiss_1999}, we now introduce a Weiss-type functional adapted to our framework.

\begin{defn}
We  define the \emph{Weiss-type functional}
\begin{equation}
\label{eq:Weiss_functional}
W_L(v,r,x_0) := \frac{1}{r^{n+2}} I_{x_0}(r) - \frac{1+s}{r^{n+3}} F_{x_0}(r),\quad\forall\, r\in \RR_+,
\end{equation}
where we recall that the function $F_{x_0}(r)$ is defined in \eqref{eq:F}.
\end{defn}

\begin{rmk}
Although, as it was pointed put in \eqref{eq:Properties_v_1}, strictly speaking the function $v_{x_0}$ satisfies the equation 
$L_a v_{x_0} =0$ in $\RR^n\times(\RR\setminus\{0\})$, in order to avoid making the notation too cumbersome we have opted for $W_L(v,r,x_0)$, instead of the heavier notation $W_{L_a}(v,r,x_0)$.  Furthermore, because the free boundary point $x_0$ is kept fixed in most of our proofs, for the sake of brevity we write $W_L(v,r)$ instead of $W_L(v,r,x_0)$. Also, when $r=1$, we write for simplicity $W_L(v)$ instead of $W_L(v,1,x_0)$.
\end{rmk}

We recall some useful identities concerning the functionals $I_{x_0}(r)$ and $F_{x_0}(r)$. The integration by parts formula together with the system of conditions \eqref{eq:Properties_v_1}--\eqref{eq:Equality_L_a} gives
\begin{equation}
\label{eq:I_normal_derivative}
I_{x_0}(r) = \int_{\partial B_r(x_0)} v_{x_0}\nabla v_{x_0} \dotprod \nu |y|^a,
\end{equation}
where $\nu$ denotes the outer unit normal to $\partial B_r(x_0)$. Differentiating \eqref{eq:I} with respect to $r$ gives
\begin{equation}
\label{eq:I_derivative}
I'_{x_0}(r) = \int_{\partial B_r(x_0)} |\nabla v_{x_0}|^2 |y|^a + \int_{\partial B'_r(x_0)} v_{x_0} h_{x_0}.
\end{equation}
From \eqref{eq:I_derivative}, and \cite[Lemma~A.7 and Identity (A.8)]{Petrosyan_Pop}, we thus obtain
\begin{align}
\label{eq:Derivative_I}
I'_{x_0}(r) &= 2\int_{\partial B_r(x_0)} |\nabla v_{x_0} \dotprod \nu|^2 |y|^a 
+ \frac{n+a-1}{r} \int_{\partial B_r(x_0)} v_{x_0} \nabla v_{x_0}\dotprod\nu |y|^a 
\\
&\qquad
- \frac{n+a-1}{r} \int_{B'_r(x_0)} v_{x_0} h_{x_0}
-\frac{2}{r} \int_{B'_r(x_0)} (x,y)\dotprod \nabla v_{x_0} h_{x_0}
+\int_{\partial B'_r(x_0)} v_{x_0} h_{x_0}.
\notag
\end{align}
We also easily obtain the derivative of the functional $F_{x_0}(r)$ in \eqref{eq:F}:
\begin{equation}
\label{eq:Derivative_F}
F'_{x_0}(r) = 2 \int_{\partial B_r(x_0)} v_{x_0} \nabla v_{x_0}\dotprod\nu |y|^a + \frac{n+a}{r} F_{x_0}(r).
\end{equation}

We next want to understand the behavior of the Weiss functional $W_L(v,r,x_0)$, as $r$ tends to $0$. We begin by proving that the functional $W_L(v,r,x_0)$ is bounded as $r$ tends to $0$, and for this purpose we make use of the following result.

\begin{lem}[Growth of $v_{x_0}$ near $x_0$]
\label{lem:Growth_v_free_boundary_point}
Let $x_0\in\Gamma(u)$. Then, there exists $C>0$ such that 
\begin{equation}
\label{eq:Growth_v_free_boundary_point}
|v_{x_0}(x,y)| \leq C|(x-x_0, y)|^{1+s},\quad\forall\, (x, y)\in\RR^{n+1}.
\end{equation}
\end{lem}

\begin{proof}
The method of proof of  \cite[Claim~2.20]{Petrosyan_Pop} can be adapted to the present setting to yield estimate \eqref{eq:Growth_v_free_boundary_point}. A more detailed proof is given in Lemma~\ref{lem:Growth_v_balls}.
\end{proof}

\begin{lem}[Boundedness of the Weiss-type functional]
\label{lem:Boundedness_Weiss}
Let $x_0\in\Gamma(u)$. Then, there exists $C, r_0>0$ such that for every $r\in (0,r_0)$ one has
\begin{equation}\label{eq:Boundedness_F}
|F_{x_0}(r)|  \leq C r^{n+3},\quad
|I_{x_0}(r)|  \leq C r^{n+2}.
\end{equation}
In particular, we obtain 
\begin{equation}
\label{eq:Boundedness_Weiss}
|W_L(v,r)| \leq C,\quad 0<r<r_0.
\end{equation}
\end{lem}

\begin{proof}
The proof of the former inequality in \eqref{eq:Boundedness_F} is an immediate consequence of the growth bound \eqref{eq:Growth_v_free_boundary_point} and of the definition \eqref{eq:F} of the functional $F_{x_0}(r)$. The growth estimate \eqref{eq:Growth_v_h_on_R_n} imply the existence of $C, r_0>0$ such that
$$
\left|\int_{B'_r(x_0)}v_{x_0} h_{x_0}\right| \leq C r^{n+1+2s},\quad 0<r<r_0.
$$
Using this estimate together with \eqref{eq:Boundedness_F}, the fact that $2s>1$, and that the functional 
$W_L(v,r)+Cr^{2s-1}$ is nondecreasing, we infer the existence of $C, r_0>0$ such that
$$
\frac{1}{r^{n+2}}\int_{B_r(x_0)} |\nabla v_{x_0}|^2 |y|^a \leq C +
W_L(v,1),\quad 0<r<r_0.
$$ 
From this estimate the latter inequality in \eqref{eq:Boundedness_F}, and \eqref{eq:Boundedness_Weiss} now follow.
\end{proof}

Analogously to \cite[Theorem~4.3]{Garofalo_Petrosyan_SmitVegaGarcia} (see also the original result by Weiss for the classical obstacle problem in \cite[Theorem~2]{Weiss_1999}), we have the following crucial monotonicity formula.

\begin{thm}[Adjusted monotonicity of the Weiss-type functional]
\label{lem:Monotonicity_Weiss}
There exist constants $C, r_0>0$ such that for all $x_0\in\Gamma(u)$ and every
$0<r<r_0$ one has:
\begin{equation}
\label{eq:Monotonicity_Weiss_more_precise}
\frac{d}{dr}\left(W_L(v,r) + Cr^{2s-1}\right) \geq 
\frac{2}{r^{n+2}} \int_{\partial B_r(x_0)} \left(\frac{(1+s)v_{x_0}}{r}-\nabla v_{x_0}\dotprod\nu\right)^2|y|^a.
\end{equation}
In particular, it follows that the function
\begin{equation}
\label{eq:Monotonicity_Weiss}
r \mapsto W_L(v,r) + Cr^{2s-1}
\end{equation}
is nondecreasing on $(0,r_0)$.
\end{thm}

\begin{proof}
From expression \eqref{eq:Weiss_functional} of the Weiss functional we obtain
\begin{equation}
\label{eq:Derivative_W}
\frac{d}{dr}W_L(v,r) = -\frac{n+2}{r^{n+3}} I_{x_0}(r) + \frac{1}{r^{n+2}} I'_{x_0}(r) + \frac{(1+s)(n+3)}{r^{n+4}} F_{x_0}(r)
-\frac{1+s}{r^{n+3}} F'_{x_0}(r).
\end{equation}
Combining identities \eqref{eq:Derivative_I} and \eqref{eq:Derivative_F}, and using the fact that $a=1-2s$, it follows that
\begin{align*}
\frac{d}{dr}W_L(v,r) &= \frac{2}{r^{n+2}} \int_{\partial B_r(x_0)} \left(\frac{(1+s)v_{x_0}}{r}-\nabla v_{x_0}\dotprod\nu\right)^2|y|^a\\
&\qquad-\frac{n+a-1}{r^{n+3}} \int_{B'_r(x_0)} v_{x_0} h_{x_0} 
- \frac{2}{r^{n+3}} \int_{B'_r(x_0)} (x,y)\dotprod \nabla v_{x_0} h_{x_0}
+ \frac{1}{r^{n+2}} \int_{\partial B'_r(x_0)} v_{x_0} h_{x_0}.
\end{align*}
Using the basic estimates \eqref{eq:Growth_v_h_on_R_n}, we obtain the upper bound
\begin{equation}
\label{eq:Error_term_W}
\begin{aligned}
&\left|\frac{n+a-1}{r^{n+3}} \int_{B'_r(x_0)} v_{x_0} h_{x_0}\right| 
+ \left|\frac{2}{r^{n+3}} \int_{B'_r(x_0)} (x,y)\dotprod \nabla v_{x_0} h_{x_0}\right|
+ \left|\frac{1}{r^{n+2}} \int_{\partial B'_r(x_0)} v_{x_0} h_{x_0} \right| \\
&\qquad\leq C r^{-2(1-s)} = \frac{C}{2s-1} \frac{d}{dr} r^{2s-1},
\end{aligned}
\end{equation}
for all $r\in (0,r_0)$ and for some number $r_0>0$ depending on $x_0$.
It follows that the inequality \eqref{eq:Monotonicity_Weiss_more_precise} for $C$ and $r\in (0,r_0)$. This, in turn, implies that the functional \eqref{eq:Monotonicity_Weiss} is nondecreasing.
\end{proof}

We can now establish a result which is analogous to \cite[Lemma~4.4]{Garofalo_Petrosyan_SmitVegaGarcia}.

\begin{lem}
\label{lem:Weiss_0_regular_points}
If $x_0\in\Gamma_{1+s}(u)$, then $W_L(v,0+) = 0$.
\end{lem}

\begin{proof}
From definition \eqref{eq:Weiss_functional} of $W_L(v,r)$, we have that
$$
W_L(v,r) = \frac{F_{x_0}(r)}{r^{n+3}} \left(r\frac{I_{x_0}(r)}{F_{x_0}(r)}-(1+s)\right),\quad\forall\, r\in \RR_+.
$$
Identities \eqref{eq:Phi_at_regular_points}, \eqref{eq:I_normal_derivative} and \eqref{eq:Derivative_F} gives us that
$$
2r\frac{I_{x_0}(r)}{F_{x_0}(r)} = \Phi^p_{x_0}(r) - (n+a),
$$
where $p\in (2,2s-1/2)$. Applying property \eqref{eq:Regular_points}, we have that
$$
\lim_{r \downarrow 0} r\frac{I_{x_0}(r)}{F_{x_0}(r)} = 1+s,
$$
and using the boundedness property \eqref{eq:Boundedness_F}, we obtain the conclusion.
\end{proof}

\subsection{Homogeneous rescalings}
\label{sec:Homogeneous_res}
To study the regularity of the free boundary in a neighborhood of regular points, we use in a fundamental way the following homogeneous rescalings of the height function $v_{x_0}(x,y)$, defined in \eqref{eq:Height_function}. For $x_0\in\Gamma_{1+s}(u)$, we define
\begin{equation}
\label{eq:Homogeneous_res}
v_{x_0,r}(x,y) := \frac{1}{r^{1+s}}v_{x_0}(x_0+rx,ry),\quad\forall\, (x,y)\in\RR^{n+1},\quad\forall\, r>0.
\end{equation}
When $x_0=O$, we write for brevity $v_r$ instead of $v_{x_0,r}$. In the sequel, we will use two main results about the sequence of homogeneous rescalings $\{v_{x_0,r}\}_{r>0}$: the convergence result in Lemma~\ref{lem:Convergence_homogeneous_res}, and the homogeneity of the limit, established in Lemma~\ref{lem:Limit_homogeneous_res}.

\begin{lem}[Convergence of the sequence of rescalings]
\label{lem:Convergence_homogeneous_res}
Let $x_0\in \Gamma_{1+s}(u)$. Any sequence of homogeneous rescalings, $\{v_{x_0,r_k}\}_{k\in\NN}$, such that $r_k\rightarrow 0$, as $k$ tends to $\infty$, contains a convergent subsequence in $C^{1+\gamma}_a(\bar B^+_{1/8})$, for some $\gamma\in (0,1)$. Any limit function, $v_{x_0,0}$, of a convergent subsequence of $\{v_{x_0,r_k}\}_{k\in\NN}$ belongs to $C^{1+\gamma}_a(\bar B^+_{1/8})$, and satisfies conditions \textup{\eqref{eq:Properties_v_1}--\eqref{eq:Equality_L_a}}, with $v_{x_0}$ replaced by $v_{x_0,0}$, and $h_{x_0}(x)$ replaced by 0.
\end{lem}

We also have 
\begin{lem}[Limit of the sequence of rescalings]
\label{lem:Limit_homogeneous_res}
Assume that the hypotheses of Lemma~\ref{lem:Convergence_homogeneous_res} hold, and let $v_{x_0,0}$ be as in Lemma~\ref{lem:Convergence_homogeneous_res}. Then  $v_{x_0,0}$  is a homogeneous function of degree $1+s$.
\end{lem}

We establish Lemmas~\ref{lem:Convergence_homogeneous_res} and \ref{lem:Limit_homogeneous_res} with the aid of several intermediate results.

\begin{lem}
\label{lem:Equation_homogeneous_res}
Let $x_0\in\Gamma_{1+s}(u)$, and $r>0$. The homogeneous rescaling $v_{x_0,r}$ satisfies the system of conditions \textup{\eqref{eq:Properties_v_1}--\eqref{eq:Equality_L_a}}, with $v_{x_0}$ replaced by $v_{x_0,r}$, and $h_{x_0}(x)$ replaced by $r^sh_{x_0}(x_0+rx)$.
\end{lem}

\begin{proof}
Direct calculations give us that
$$
L_a v_{x_0,r}(x,y) = r^s L_a v_{x_0}(x_0+rx,ry),\quad\forall\, (x,y)\in\RR^{n+1}.
$$
and so, the conclusion of the lemma follows immediately.
\end{proof}

The proof of Lemma~\ref{lem:Convergence_homogeneous_res} is based on the uniform a priori local Schauder estimates in the H\"older space of functions $C^{1+\alpha}_a(\bar B^+_{1/8})$, defined in \eqref{eq:Schauder_space}.

\begin{lem}[Uniform Schauder estimates]
\label{lem:Schauder_homogeneous_res}
Let $x_0\in\Gamma_{1+s}(u)$. Then, there exist cons\-tants $C, r_0>0$, and $\alpha\in (0,1)$, such that
\begin{equation}
\label{eq:Schauder_estimates_homogeneous_res}
\|v_{x_0,r}\|_{C^{1+\alpha}_a(\bar B^+_{1/8})} \leq C,\quad\forall\, r\in (0,r_0).
\end{equation}
\end{lem}

\begin{proof}
Because $x_0$ belongs to $\Gamma_{1+s}(u)$, and the homogeneous rescalings $\{v_{x_0,r}\}_{r>0}$ satisfy the conclusion of Lemma~\ref{lem:Equation_homogeneous_res}, it follows that the hypotheses of \cite[Lemma~2.17]{Petrosyan_Pop} are verified, and so there are positive constants, $\alpha\in (0,1)$, $C$ and $r_0$, such that estimate \eqref{eq:Schauder_estimates_homogeneous_res} holds, for all $r\in(0,r_0)$.
\end{proof}

We can now give the proof of Lemma~\ref{lem:Convergence_homogeneous_res} with the aid of Lemma~\ref{lem:Schauder_homogeneous_res}.
\begin{proof}[Proof of Lemma~\ref{lem:Convergence_homogeneous_res}]
The Schauder estimate \eqref{eq:Schauder_estimates_homogeneous_res}
and Arzel\'a-Ascoli Theorem implies that the sequence of rescalings, $\{v_{x_0,r_k}\}_{k\in\NN}$, contains a convergent subsequence in any space $C^{1+\gamma}_a(\bar B^+_{1/8})$, for all $\gamma\in (0,\alpha)$, where $\alpha\in (0,1)$ is the constant appearing in the conclusion of Lemma~\ref{lem:Schauder_homogeneous_res}. From Lemma~\ref{lem:Equation_homogeneous_res}, it follows that any limit function of a convergence subsequence is a solutions to the system of conditions \eqref{eq:Properties_v_1}--\eqref{eq:Equality_L_a}, with $h_{x_0}(x)$ replaced by $0$.
\end{proof}

Next, we give the proof of Lemma~\ref{lem:Limit_homogeneous_res}, using the monotonicity property of the Weiss functional established in Lemma~\ref{lem:Monotonicity_Weiss}.

\begin{proof}[Proof of Lemma~\ref{lem:Limit_homogeneous_res}]
Let $r_0$ be the positive constant in the hypotheses of Lemma~\ref{lem:Monotonicity_Weiss}, and let $0<R_1<R_2<r_0$. We apply inequality \eqref{eq:Monotonicity_Weiss_more_precise} to $v_{x_0}$ and integrate over the interval $(r_kR_1,r_kR_2)$, obtaining that
\begin{align*}
&W_L(v_{x_0},r_k R_2) - W_L(v_{x_0},r_k R_1) + Cr_k^{2s-1}(R_2^{2s}-R_1^{2s})\\
&\qquad\quad \geq \int_{r_kR_1}^{r_kR_2} \frac{2}{r^{n+4}}\int_{\partial B_r}\left[(1+s)v_{x_0}(x_0+x,y)-\nabla v_{x_0}(x_0+x,y)\dotprod (r\nu)\right]^2|y|^a\, dr\\
&\qquad\quad = \int_{R_1}^{R_2} \frac{2}{r_k^{n+3}r^{n+4}}\int_{\partial B_{rr_k}}\left[(1+s)v_{x_0}(x_0+x,y)-\nabla v_{x_0}(x_0+x,y)\dotprod (rr_k\nu)\right]^2|y|^a\, dr,
\end{align*}
where $\nu$ denotes the outer unit normal vector to the spheres $\partial B_r$ and $\partial B_{rr_k}$. Using the definition of the homogeneous rescalings \eqref{eq:Homogeneous_res}, and that of the Weiss functional \eqref{eq:Weiss_functional}, we obtain in the preceding inequality,
\begin{align*}
&W_L(v_{x_0}, r_kR_2) - W_L(v_{x_0}, r_kR_1) + Cr_k^{2s-1}(R_2^{2s}-R_1^{2s})\\
&\qquad\quad \geq \int_{R_1}^{R_2} \frac{2r_k^{n+a+2(1+s)}}{r_k^{n+3}}\frac{1}{r^{n+4}}\int_{\partial B_{r}}\left[(1+s)v_{x_0,r_k}-\nabla v_{x_0,r_k}\dotprod (r\nu)\right]^2|y|^a\, dr\\
&\qquad\quad = \int_{R_1}^{R_2} \frac{2}{r^{n+4}}\int_{\partial B_{r}}\left[(1+s)v_{x_0,r_k}-\nabla v_{x_0,r_k}\dotprod (r\nu)\right]^2|y|^a\, dr.
\end{align*}
Letting now $k$ tend to $\infty$, and using the fact that $2s>1$, it follows from Lemma~\ref{lem:Weiss_0_regular_points} that the left-hand side in the preceding inequality tends to 0. Applying also Lemma~\ref{lem:Convergence_homogeneous_res} to the right-hand side of the preceding inequality, we see that 
\begin{align*}
0 & \geq \int_{R_1}^{R_2} \frac{2}{r^{n+4}}\int_{\partial B_{r}}\left[(1+s)v_{x_0,0}-\nabla v_{x_0,0}\dotprod (r\nu)\right]^2|y|^a\, dr.
\end{align*}
Because the positive constants $R_1<R_2$ are arbitrarily chosen in the interval $(0,r_0)$, it follows that $\nabla v_{x_0,0}\dotprod (r\nu) =(1+s)v_{x_0,0} $ on $\partial B_r$, for all $r\in (0,r_0)$, and so the limit function $v_{x_0,0}$ is homogeneous of degree $1+s$. This completes the proof.
\end{proof}

\section{An epiperimetric inequality}
\label{S:epi}
In this section we establish a generalization  of the epiperimetric inequality obtained by Weiss for the classical obstacle problem in the context of the obstacle problem for the fractional Laplacian with drift. Our main result, Theorem~\ref{T:epi}, is tailor made for analyzing regular free boundary points.

Let $x_0\in\Gamma_{1+s}(u)$. For the purpose of this section, we can assume without loss of generality that $x_0=0$. Following \cite[p. 27]{Weiss_1999} (see  also \cite[Definition 6.1]{Garofalo_Petrosyan_SmitVegaGarcia}), we define a version of the boundary adjusted Weiss energy adapted to our framework.
\begin{defn}[Boundary adjusted Weiss energy]
\label{D:bae}
 Given $v\in H^1(B_1,|y|^a)$, we next introduce the \emph{boun\-dary adjusted energy} as the Weiss type functional defined in \eqref{eq:Weiss_functional}, 
with $r=1$ and zero obstacle, i.e.,
\begin{equation}
\label{eq:Boundary_Weiss_functional}
W(v):= W(v,1)=\int_{B_1}|\nabla v|^2|y|^a-(1+s)\int_{\partial B_1}v^2|y|^a.
\end{equation}
\end{defn}

We now consider the function 
$$
\hat{v}_0(x,y)=\left(x_n+\sqrt{x_n^2+y^2}\right)^s\left(x_n-s\sqrt{x_n^2+y^2}\right).
$$
The function $\hat{v}_0$ belongs to $\cH_{1+s}$, and so it is a $(1+s)$-homogeneous global solution of the obstacle problem for the fractional Laplacian \eqref{eq:Obstacle_problem_without_drift} with zero obstacle function. The following is the central result of this section, which is a generalization of \cite[Theorem~1]{Weiss_1999} to the setting of our article. This result adapts \cite[Theorem~6.3]{Garofalo_Petrosyan_SmitVegaGarcia} to the context of the present work. 

\begin{thm}[Epiperimetric inequality]\label{T:epi}  There exists
  $\kappa\in (0,1)$ and $\delta\in (0,1)$ such that if $w\in H^1(B_1,|y|^a)$ is a homogeneous function of degree $(1+s)$ such
  that $w\ge 0$ on $B_1'$ and $\|w-\hat{v}_0\|_{H^1(B_1,|y|^a)}\le \delta$, then
  there exists $\tilde{w}\in H^1(B_1,|y|^a)$ such that $\tilde{w}=w$ on $\partial B_1$,
  $\tilde{w}$ is nonnegative on $B_1'$ and 
\[
W(\tilde{w})\le (1-\kappa)W(w).
\]
\end{thm}

\begin{rmk}\label{R:weiss}
We observe explicitly that if $v$ is a solution to the obstacle problem \eqref{eq:Obstacle_problem_without_drift} with zero obstacle, and $v$ belongs to $\cH_{1+s}$, then we can rewrite
\[
\int_{B_1}|\nabla v|^2 |y|^a =\int_{\partial B_1}v \nabla v \dotprod\nu |y|^a =(1+s)\int_{\partial B_1}v^2 |y|^a,
\]
which implies that $W(v) = 0$. In the preceding identity, $\nu$ denotes the outer unit normal to $\partial B_1$.
\end{rmk}

\begin{proof}[Proof of Theorem~\ref{T:epi}]
We argue by contradiction and assume that the result does not
hold. Then, there exist sequences of real numbers $\kappa_m\rightarrow
0$ and $\delta_m\rightarrow 0$, and functions $w_m\in H^1(B_1,|y|^a)$,
homogeneous of degree $(1+s)$, such that $w_m\ge 0$ on $B_1'$
and
\begin{equation}\label{ate0}
\|w_m-\hat{v}_0\|_{H^1(B_1,|y|^a)}\le \delta_m,
\end{equation}
but such that, for every  $\tilde{w}_m\in H^1(B_1,|y|^a)$ with the properties that $\tilde{w}_m\ge 0$ on
$B_1'$, and $\tilde{w}_m = w_m$ on $\partial B_1$, we have that
\begin{equation}\label{ate}
W(\tilde{w}_m) > (1-\kappa_m)W(w_m).
\end{equation}
With such an assumption in place we start by observing that there exists
\[
g_m=a_m\left( x \dotprod e_m + \sqrt{ (x \dotprod e_m)^2+y^2}\right)^s
\left(x \dotprod e_m -s\sqrt{(x \dotprod e_m)^2+y^2}\right)
\]
belonging to the space of homogeneous functions $\cH_{1+s}$, which
achieves the minimum distance from $w_m$ to $\cH_{1+s}$, that is
\[
\|w_m-g_m\|_{H^1(B_1,|y|^a)}=\inf_{g\in \cH_{1+s}}\|w_m-g\|_{H^1(B_1,|y|^a)}.
\]
Indeed, this follows from the simple fact that the set $\cH_{1+s}$ is locally compact.
Combining this inequality with \eqref{ate0} we deduce that
$$
\|g_m-\hat{v}_0\|_{H^1(B_1,|y|^a)}\le 2\delta_m,
$$
and, as a consequence, we must have that $e_m\rightarrow e^n$ and $a_m\rightarrow 1$, as $m$ tends to $\infty$, where $e^n\in\RR^n$ denotes the unit vector having all coordinates zero, except for the $n$-th coordinate. Hence,
\[
\left\|\frac{ (w_m-g_m)}{a_m}\right\|_{H^1(B_1,|y|^a)}\le \frac{\delta_m}{a_m} \to 0,\quad\text{as } m\rightarrow \infty.
\]
If we rename $\frac{w_m}{a_m}$ by $w_m$, and $\frac{\delta_m}{a_m}$ by $\delta_m$, and rotate
$\RR^{n}$ to send $e_m$ to $e^n$, the renamed function $w_m$ is homogeneous of degree $(1+s)$, nonnegative on $B_1'$, and satisfies
\begin{equation}\label{ate3}
\inf_{g\in \cH_{1+s}}\|w_m-g\|_{H^1(B_1,|y|^a)}=\|w_m-\hat{v}_0\|_{H^1(B_1,|y|^a)}\le \delta_m.
\end{equation}
Moreover, inequality \eqref{ate} still holds for the renamed functions $w_m$,
because of the scaling property $W(tw)=t^2 W(w)$, and the invariance of
$W(w)$ under rotations in $\RR^{n}$. 

We note explicitly that \eqref{ate} implies in particular that
$w_m\neq \hat{v}_0$ for every $m\in \mathbb N$, as $W(\hat{v}_0)=0$, by Remark~\ref{R:weiss}.
Thus we can assume without loss of generality that
\begin{equation}\label{thetam}
\delta_m = \|w_m - \hat{v}_0\|_{H^1(B_1,|y|^a)}>0.
\end{equation}
We now want to rewrite \eqref{ate} in a slightly different way, using the properties of function $\hat{v}_0$. Given $\phi\in H^1(B_1,|y|^a)$, we consider the first variation of $W$ at $\hat{v}_0$ in the direction of $\phi$,
\begin{equation}\label{apples}
\delta W(\hat{v}_0)(\phi) :=\int_{B_1}2 \nabla \hat{v}_0 \dotprod \nabla\phi |y|^a -(1+s)\int_{\partial B_1}2\hat{v}_0\phi|y|^a,
\end{equation}
where the boundary integral in \eqref{apples} and thereafter is interpreted in the sense of traces. To compute $\delta W(\hat{v}_0)(\phi)$, we rewrite the first integral in the
right-hand side of \eqref{apples} as 
\begin{align*}
\int_{B_1}2 \nabla \hat{v}_0 \dotprod\nabla \phi |y|^a
&= -4\int_{B_1'}\phi\lim\limits_{y \downarrow 0} |y|^a \partial_y \hat{v}_0 +\int_{\partial B_1}2\phi \nabla \hat{v}_0 \dotprod\nu |y|^a,
\end{align*}
where we used the fact that the function $\hat v_0$ is symmetric with respect to the hyperplane $\{y=0\}$. In the preceding identity, $\nu$ denotes the unit outer normal to $\partial B_1$. Because the function $\hat{v}_0$ is homogeneous of degree $(1+s)$,  Euler's formula gives us that
\[
\int_{B_1}2 \nabla \hat{v}_0 \dotprod\nabla \phi |y|^a =  -4\int_{B_1'}\phi\lim\limits_{y \downarrow 0}|y|^a \partial_y\hat{v}_0 + (1+s) \int_{\partial B_1}2\phi \hat{v}_0 |y|^a.
\]
We conclude that
\begin{equation}\label{apples0}
\delta W(\hat{v}_0)(\phi) =  -4\int_{B_1'}\phi\lim\limits_{y \downarrow 0}|y|^a \partial_y \hat{v}_0.
\end{equation}
For any function $\tilde{w}_m \in H^1(B_1, |y|^a)$ with the properties that $\tilde{w}_m\ge 0$ on $B_1'$, and $\tilde{w}_m = w_m$ on $\partial B_1$, by plugging in $\phi=\tilde{w}_m-\hat{v}_0$ into identities \eqref{apples} and \eqref{apples0}, we obtain that
\begin{align*}
W(\tilde{w}_m)&= W(\tilde{w}_m)- W(\hat{v}_0)- \delta W(\hat{v}_0)(\tilde{w}_m-\hat{v}_0)  
-4\int_{B_1'}(\tilde{w}_m-\hat{v}_0)\lim\limits_{y \downarrow 0}|y|^a \partial_y \hat{v}_0\\
 &= \int_{B_1}|\nabla (\tilde{w}_m-\hat{v}_0)|^2|y|^a
 -(1+s)\int_{\partial B_1}(\tilde{w}_m-\hat{v}_0)^2|y|^a
  - 4\int_{B_1'}(\tilde{w}_m-\hat{v}_0)\lim\limits_{y \downarrow 0}|y|^a \partial_y \hat{v}_0,
\end{align*}
where we have used in the first identity, the fact that that $W(\hat{v}_0)=0$, by Remark~\ref{R:weiss}. Using a similar identity for $W(w_m)$, we can rewrite inequality \eqref{ate} as
\begin{multline}\label{theone}
(1-\kappa_m)\left[\int_{B_1}|\nabla (w_m-\hat{v}_0)|^2|y|^a-(1+s)\int_{\partial B_1}(w_m-\hat{v}_0)^2|y|^a- 4\int_{B_1'}(w_m-\hat{v}_0)\lim\limits_{y \downarrow 0}|y|^a \partial_y\hat{v}_0\right]
\\<\int_{B_1}|\nabla (\tilde{w}_m-\hat{v}_0)|^2|y|^a-(1+s)\int_{\partial B_1}(\tilde{w}_m-\hat{v}_0)^2|y|^a
- 4\int_{B_1'}(\tilde{w}_m-\hat{v}_0)\lim\limits_{y \downarrow 0}|y|^a \partial_y\hat{v}_0.
\end{multline}
Inequality \eqref{theone} will play a key role in the proof of the epiperimetric inequality, and it will be used repeatedly.

Let us introduce the normalized functions 
$$
\hat{w}_m=\frac{w_m-\hat{v}_0}{\delta_m},\quad\forall\,m\in\NN.
$$ 
By identity \eqref{thetam}, we have 
\begin{equation}
\label{eq:Uniform_H_1_norm}
\|\hat{w}_m\|_{H^1(B_1,|y|^a)}=1\quad\,\forall\, m\in \NN.
\end{equation}
By the weak compactness of the unit sphere in $H^1(B_1,|y|^a)$, we can assume that 
$$
\hat{w}_m \to \hat{w}\quad\text{weakly in }H^1(B_1,|y|^a),\quad\text{as } m\rightarrow\infty.
$$
By the compactness of the Sobolev embedding and traces operator from $H^1(B_1,|y|^a)$ into the Sobolev space $L^2(B_1,|y|^a)$, $L^2(B_1')$, $L^2(\partial B_1,|y|^a)$, we may assume that
$$
\hat{w}_m\to\hat{w}\quad\text{strongly in $L^2(B_1,|y|^a)$, $L^2(B_1')$, and $L^2(\partial B_1,|y|^a)$},
\quad\text{as } m\rightarrow\infty. 
$$
See \cite[Theorem~2.8]{Nekvinda_1993} for the boundedness of the trace operator from $H^1(B_1, |y|^a)$ into $L^2(B'_1)$, and see \cite[Lemma~A.25]{Daskalopoulos_Feehan_statvarineqheston} for the boundedness of the trace operator from $H^1(B_1, |y|^a)$ into $L^2(\partial B_1, |y|^a)$.

We then make the following

\begin{claim}
\label{claim:hat_w}
The limit function $\hat w$ satisfies the following properties:
\begin{itemize}
\item[(i)] $\hat{w}\equiv 0$;
\item[(ii)] $\hat{w}_m\rightarrow 0$ strongly in $H^1(B_1,|y|^a)$, as $m\rightarrow\infty$.
\end{itemize}
\end{claim}

Note that property (ii) will give us a contradiction with condition \eqref{eq:Uniform_H_1_norm}. Hence, the theorem will follow once we prove the claim. In what follows, we denote 
$$
\Lambda=\Lambda(\hat{v}_0)=\{(x,0)\in \RR^n\times\{0\}\mid \hat{v}_0(x,0)=0\},
$$
the coincidence set of the function $\hat{v}_0$. 

\begin{proof}[Proof of Claim~\ref{claim:hat_w}]
We organize the proof into several steps.

\setcounter{step}{0}
\begin{step}
\label{step:1}
We start by showing that there is a positive constant, $C$, such that 
\begin{equation}\label{sh1}
\left\|\frac{w_m}{\delta_m^2} \lim\limits_{y\downarrow 0}|y|^a \partial_y\hat{v}_0\right\|_{L^1(B_1')} \le
C,\quad\forall\, m\in \mathbb N.
\end{equation}
To this end, we pick a function $\eta\in W^{1,\infty}_0(B_1)$, 
such that $0<\eta\le 1$, and define 
$$
\tilde{w}_m=(1-\eta)w_m+\eta \hat{v}_0.
$$
Then, it is clear that the function $\tilde w_m$ satisfies the properties:
$$
\tilde{w}_m=w_m \text{ on } \partial B_1,
\quad
\tilde{w}_m\ge 0 \text{ on } B_1',
\quad\text{and}\quad 
\tilde{w}_m-\hat{v}_0 = (1-\eta)(w_m-\hat{v}_0).
$$
We can thus apply inequality \eqref{theone} to the function $\tilde{w}_m$, obtaining 
\begin{align*}
&(1-\kappa_m)\left(\int_{B_1}|\nabla (w_m-\hat{v}_0)|^2|y|^a-(1+s)\int_{\partial B_1}(w_m-\hat{v}_0)^2|y|^a
-4\int_{B_1'}w_m  \lim\limits_{y\downarrow 0}|y|^a \partial_y\hat{v}_0\right)\\
&\qquad< \int_{B_1}|\nabla ((1-\eta)(w_m-\hat{v}_0))|^2|y|^a-(1+s)\int_{\partial B_1}(1-\eta)^2(\hat{v}_0-w_m)^2|y|^a\\
&\qquad\qquad-4\int_{B_1'}(1-\eta)(w_m- \hat{v}_0)  \lim\limits_{y\downarrow 0}|y|^a \partial_y\hat{v}_0\\
&\qquad = \int_{B_1}\left[(1-\eta)^2|\nabla (w_m-\hat{v}_0)|^2+|\nabla \eta|^2(w_m-\hat{v}_0)^2-2(1-\eta)(w_m-\hat{v}_0)\nabla \eta \dotprod\nabla ( w_m-\hat{v}_0)\right]|y|^a\\
&\qquad\qquad-(1+s)\int_{\partial B_1}(1-\eta)^2(\hat{v}_0-w_m)^2|y|^a-4\int_{B_1'}(1-\eta)w_m \lim\limits_{y\downarrow 0}|y|^a \partial_y\hat{v}_0.
\end{align*}
Dividing by $\delta_m^2$, rearranging terms and using property \eqref{eq:Uniform_H_1_norm},
it follows that
\begin{align*}
4 \int_{B_1'} ( \kappa_m-\eta) \frac{w_m}{\delta_m^2} \lim\limits_{y\downarrow 0}|y|^a \partial_y\hat{v}_0
&\le -(1-\kappa_m)\left(\int_{B_1}|\nabla \hat{w}_m|^2 |y|^a
-(1+s)\int_{\partial B_1}\hat{w}_m^2|y|^a\right)\\
&\qquad +\int_{B_1}\left[(1-\eta)^2|\nabla \hat{w}_m|^2+|\nabla
  \eta|^2\hat{w}_m^2-2(1-\eta)\hat{w}_m \nabla \eta\dotprod\nabla \hat{w}_m\right]|y|^a\\
&\qquad  -(1+s)\int_{\partial B_1}(1-\eta)^2\hat{w}_m^2|y|^a\\
&\le C,
\end{align*}
where $C$ is a positive constant, independent of $m\in \mathbb N$. At this point, we choose $\eta(x)=\tilde{\eta}(|x|)$, and let  
\[
0 < \varepsilon = \int_0^1 \tilde \eta(r) r^{n+1} dr.
\]
Since $\kappa_m \to 0$, as $m\to \infty$, possibly by passing to a
subsequence, we can assume without loss of generality that $\kappa_m \le \frac \varepsilon 2(n+2)$, for
every $m\in \mathbb N$. With such a choice, we have that
\[
\int_0^1(\tilde{\eta}(r)-\kappa_m)r^{n+1} dr \ge  \frac{\varepsilon}{2},\quad \forall\,m\in \mathbb N.
\]
Using the fact that $w_m$ and $\hat{v}_0$ are homogeneous functions of degree $1+s$, we obtain that
\begin{align*}
C&\ge 4\int_{B_1'}(\kappa_m-\eta)\frac{w_m}{\delta_m^2} \lim\limits_{y\downarrow 0}|y|^a \partial_y\hat{v}_0 = 4\left(\int_0^1(\kappa_m-\tilde{\eta}(r))r^{n+1} dr\right)\int_{\partial B_1'}\frac{w_m}{\delta_m^2}\lim\limits_{y\downarrow 0}|y|^a \partial_y\hat{v}_0 \\
&\ge 2\varepsilon \int_{\partial B_1'}\frac{w_m}{\delta_m^2}\left(-\lim\limits_{y\downarrow 0}|y|^a \partial_y\hat{v}_0\right),
\end{align*}
which, again by the homogeneity of $w_m$ and $\hat{v}_0$, the fact that
$w_m\ge 0$ on $B_1'$ and the fact that 
\begin{equation}
\label{eq:Limit_v_y_at_0}
\lim_{y\downarrow 0}|y|^a \partial_y\hat{v}_0\le 0\quad\text{on } B_1',
\end{equation}

proves inequality \eqref{sh1}.
\end{step}

\begin{step}
\label{step:2}
We start by showing that
\begin{equation}\label{deltaw}
L_a \hat w = 0\quad \text{on }\  B_1\setminus\Lambda.
\end{equation}
To establish property \eqref{deltaw}, it is sufficient to show that, for any ball $B$, such that its concentric double $2B\Subset B_1\setminus \Lambda$, and for any function $\phi\in H^1(B, |y|^a)$, such that $\phi - \hat{w} \in H^1_0(B, |y^a|)$, that is $\phi=\hat w$ in the trace sense on $\partial B$, we have that
\[
\int_B |\nabla \hat{w}|^2 |y|^a\le \int_B |\nabla \phi|^2|y|^a.
\]
To begin, we fix a function $\phi \in L^{\infty}(B_1)\cap H^1(B, |y|^a)$, and we consider 
\[
\tilde{w}_m=\eta(\hat{v}_0+\delta_m\phi)+(1-\eta)w_m, 
\]
where $\eta\in C^{\infty}_0(B_1\setminus \Lambda)$ is such that $0\le\eta\le
1$. Notice that on $\partial B_1$, we have that $\tilde{w}_m=w_m$, and because $\phi\in
L^{\infty}(B_1)$ and $\eta\in C^{\infty}_0(B_1\setminus \Lambda)$, for
$m$ large enough, we have $\tilde{w}_m$ is nonnegative on $B_1'$. For such sufficiently
large $m$, we can thus use the function $\tilde{w}_m$ in inequality \eqref{theone}, and dividing by $\delta_m^2$, we obtain
\begin{align*}
&(1-\kappa_m)\left(\int_{B_1}|\nabla \hat{w}_m|^2|y|^a-(1+s)\int_{\partial B_1}\hat{w}_m^2|y|^a-4\int_{B_1'}\frac{w_m}{\delta_m^2}\lim\limits_{y \downarrow 0}|y|^a \partial_y\hat{v}_0\right)\\
&\qquad<\int_{B_1}\left[|\nabla (\eta\phi)|^2+|\nabla ((1-\eta)\hat{w}_m)|^2+2\nabla(\eta\phi)\dotprod\nabla((1-\eta)\hat{w}_m)\right]|y|^a
\\
&\qquad\qquad -(1+s)\int_{\partial B_1}((1-\eta)\hat{w}_m+\eta\phi)^2|y|^a
-4\int_{B_1'}(1-\eta)\frac{w_m}{\delta_m^2}\lim\limits_{y \downarrow 0}|y|^a \partial_y\hat{v}_0\\
&\qquad =\int_{B_1}\left[|\nabla (\eta\phi)|^2+|\nabla
  ((1-\eta)\hat{w}_m)|^2+2\nabla(\eta\phi)\dotprod \nabla((1-\eta)\hat{w}_m)\right]|y|^a\\
&\qquad\qquad -(1+s)\int_{\partial B_1}\hat{w}_m^2|y|^a
- 4\int_{B_1'}(1-\eta)\frac{w_m}{\delta_m^2}\lim\limits_{y \downarrow 0}|y|^a \partial_y\hat{v}_0,\\
\end{align*}
where in the last line we used the fact that $\eta\in C^{\infty}_0(B_1\setminus \Lambda)$. Using property \eqref{eq:Limit_v_y_at_0} and that $w_m$ is nonnegative on $B'_1$, the preceding inequality implies
\begin{align*}
\int_{B_1}|\nabla \hat{w}_m|^2|y|^a&< \kappa_m\int_{B_1}|\nabla \hat{w}_m|^2|y|^a+(1+s)(1-\kappa_m)\int_{\partial B_1}\hat{w}_m^2|y|^a+\\
&\qquad+\int_{B_1}\left[|\nabla (\eta\phi)|^2+|\nabla((1-\eta)\hat{w}_m)|^2+2\nabla(\eta\phi)\dotprod\nabla((1-\eta)\hat{w}_m)\right]|y|^a \\
&\qquad-(1+s)\int_{\partial B_1}\hat{w}_m^2|y|^a-4\kappa_m\int_{B_1'}\frac{w_m}{\delta_m^2}\lim\limits_{y \downarrow 0}|y|^a \partial_y\hat{v}_0\\
&=\kappa_m\int_{B_1}|\nabla \hat{w}_m|^2|y|^a-(1+s)\kappa_m\int_{\partial B_1}\hat{w}_m^2|y|^a\\
&\qquad+\int_{B_1}\left[|\nabla (\eta\phi)|^2+|\nabla((1-\eta)\hat{w}_m)|^2+2\nabla(\eta\phi)\dotprod\nabla((1-\eta)\hat{w}_m)\right]|y|^a\\
&\qquad-4\kappa_m\int_{B_1'}\frac{w_m}{\delta_m^2}\lim\limits_{y \downarrow 0}|y|^a \partial_y\hat{v}_0.
\end{align*}
Thus, we can find a positive constant, $C$, independent of $m\in\NN$, such that
\begin{align*}
\int_{B_1}|\nabla \hat{w}_m|^2|y|^a&<
C\kappa_m+\int_{B_1}\left[|\nabla (\eta\phi)|^2+|\nabla((1-\eta)\hat{w}_m)|^2+2\nabla(\eta\phi)\dotprod\nabla((1-\eta)\hat{w}_m)\right]|y|^a,
\end{align*}
which yields
\begin{align*}
\int_{B_1}(1-(1-\eta)^2)|\nabla \hat{w}_m|^2|y|^a&\le C\kappa_m+\int_{B_1}\bigl[|\nabla (\eta\phi)|^2+\hat{w}_m^2|\nabla\eta|^2\\&\qquad-2(1-\eta)\hat{w}_m\nabla\eta\dotprod\nabla \hat{w}_m+2\nabla(\eta\phi)\dotprod\nabla((1-\eta)\hat{w}_m)\bigr]|y|^a.
\end{align*}
Passing to the limit $m\rightarrow\infty$, we obtain
\begin{align*}
\int_{B_1}(1-(1-\eta)^2)|\nabla \hat{w}|^2|y|^a
&\le \int_{B_1}\bigl[|\nabla (\eta\phi)|^2+\hat{w}^2|\nabla\eta|^2\\
&\qquad-2(1-\eta)\hat{w}\nabla\eta\dotprod\nabla \hat{w}+2\nabla(\eta\phi)\dotprod\nabla((1-\eta)\hat{w})\bigr]|y|^a.
\end{align*}
Notice that 
\begin{align*}
\int_{B_1}|\nabla(\eta\phi+(1-\eta)\hat{w})|^2|y|^a&=\int_{B_1}\left[|\nabla(\eta\phi)|^2+|\nabla((1-\eta)\hat{w})|^2+2 \nabla(\eta\phi)\dotprod\nabla((1-\eta)\hat{w})\right]|y|^a\\
&=\int_{B_1}\bigl[|\nabla(\eta\phi)|^2+\hat{w}^2|\nabla\eta|^2+(1-\eta)^2|\nabla \hat{w}|^2-2\hat{w}(1-\eta) \nabla \hat{w} \dotprod\nabla \eta\\
&\qquad+2 \nabla(\eta\phi)\dotprod\nabla((1-\eta)\hat{w})\bigr]|y|^a.
\end{align*}
Hence the preceding inequalities give us that
\[
\int_{B_1}|\nabla \hat{w}|^2|y|^a\le \int_{B_1}|\nabla(\eta\phi+(1-\eta)\hat{w})|^2|y|^a.
\]
By approximation, we can remove the condition that $\phi$ belongs to $L^{\infty}(B_1)$,
and by considering open balls, $B\Subset B_1\setminus \Lambda$, we may
choose the function $\eta$ such that $\eta=1$ in $B$, and $\phi=\hat{w}$ outside $B$. This gives us
\[
\int_{B_1}|\nabla \hat{w}|^2|y|^a\le \int_B|\nabla \phi|^2|y|^a+\int_{B_1\setminus B}|\nabla \hat{w}|^2|y|^a,
\]
and so we obtain that
\[
\int_B |\nabla \hat{w}|^2 |y|^a\le \int_B |\nabla\phi|^2|y|^a,
\]
which proves that $L_a \hat{w}=0$ in $B$.
\end{step}

\begin{step}
\label{step:3}
We next want to prove that
\begin{equation}\label{step3}
\hat w= 0\quad \mathcal{H}^{n}\text{-a.e.\ in}\ \Lambda.
\end{equation}
We note that the function $\hat{v}_0$ satisfies the property that
$$
\lim_{y\downarrow 0}|y|^a \partial_y \hat{v}_0< 0,\quad\forall\, (x,0)\in \operatorname{int}(\Lambda).
$$
Therefore, given a subset $\omega \Subset \operatorname{int}(\Lambda)$, there exists a positive constant, $C_\omega$, such that 
$$
\left|\lim_{y\downarrow 0}|y|^a  \partial_y \hat{v}_0 \right|\ge C_\omega,\quad\forall\, (x,0)\in \omega.
$$
At points $(x,0)\in \operatorname{int}(\Lambda)$, we can thus write
\[
\hat{w}_m = \frac{w_m - \hat{v}_0}{\delta_m} = \frac{w_m}{\delta_m^2} \left(\lim\limits_{y\downarrow 0}|y|^a \partial_y \hat{v}_0 \right) \frac{\delta_m}{\lim\limits_{y\downarrow 0} |y|^a  \partial_y \hat{v}_0}.
\]
This gives
\[
\int_\omega |\hat{w}_m| \le \frac{\delta_m}{C_\omega} \int_\omega \frac{w_m}{\delta_m^2} \left|\lim\limits_{y\downarrow 0}|y|^a \partial_y \hat{v}_0\right|\le \frac{C\delta_m}{C_\omega},
\]
where in the last inequality we have used property \eqref{sh1}. Since $\delta_m \to 0$, we conclude that
$\|\hat{w}_m\|_{L^1(\omega)} \to 0$, as $m$ tends to $\infty$. By the arbitrariness of $\omega\Subset \operatorname{int}(\Lambda)$, we infer that
\begin{equation*}
\hat{w}_m(x,0)\to 0,\quad\mathcal H^{n}\text{-a.e.}\ (x,0)\in \Lambda,\quad\text{as } m\rightarrow\infty,
\end{equation*}
which proves identity \eqref{step3}.
\end{step}

\begin{step}[Proof of property (i)]
We next show that
\begin{equation}\label{weakly}
\hat{w}_m \to 0\quad\text{weakly in}\ H^1(B_1, |y|^a),\quad\text{as } m\rightarrow\infty,
\end{equation}
or, equivalently, that $\hat{w}=0$. We begin by observing that, since the functions $\hat{w}_m$'s are homogeneous of degree 
$1+s$, their weak limit $\hat{w}$ is also homogeneous of degree
$1+s$. Combining this observation with the results proved in Steps~\ref{step:2} and \ref{step:3}, it follows that the limit function $\hat{w}$ satisfies the following properties:
\begin{itemize}
\item[(i)] $L_a \hat{w}=0$ on $B_1\setminus\Lambda$;
\item[(ii)] $\hat{w}=0$ $\mathcal{H}^{n}$-a.e.\ on $\Lambda$;
\item[(iii)] $\hat{w}$ is homogeneous of degree $1+s$.
\end{itemize}
By Lemma~\ref{lem:Asymptotic_expansion_homogeneous_sol} we conclude that, if we define
\[
U_0(x,y)=\left(x_n+\sqrt{x_n^2+y^2}\right)^s,
\]
then there exist constants $c_0, \ldots, c_{n-1}$ such that
\begin{align*}
\hat{w}&=c_0\hat{v}_0+\sum_{j=1}^{n-1} c_j x_j U_0.
\end{align*}
We next show that all constants $c_j=0$, for all $j=1,\ldots,n$. To simplify
the notation, in the following lines, we write $\|\cdot\|=\|\cdot\|_{H^1(|y|^a,B_1)}$, and we let $\langle\cdot,\cdot\rangle$ denote the inner product in $H^1(B_1,|y|^a)$. Using property \eqref{ate3}, we have that
\[
\|w_m-g\|^2\ge \|w_m-\hat{v}_0\|^2\quad\forall\, g\in \cH_{1+s},
\]
and recalling that $\hat{w}_m=\frac{w_m-\hat{v}_0}{\delta_m}$, we can write the preceding inequality as
\[
\|\delta_m\hat{w}_m+\hat{v}_0-g\|^2\ge \|\delta_m \hat{w}_m\|^2,
\]
or
\[
2\delta_m \langle\hat{w}_m,\hat{v}_0-g\rangle +\|\hat{v}_0-g\|^2\ge 0.
\]
Therefore, it follows that
\begin{equation}\label{inequality}
 \langle \hat{w}_m, g-\hat{v}_0\rangle \le \frac{\|\hat{v}_0-g\|^2}{2\delta_m}.
\end{equation}
Applying this to $g=(1+\delta_m^2)\hat{v}_0$, we obtain
\[
 \langle\hat{w}_m, \hat{v}_0\rangle \le \frac{\delta_m}{2}\|\hat{v}_0\|^2.
\]
Letting $m\rightarrow \infty$, we arrive at
\[
 \langle\hat{w}, \hat{v}_0\rangle = c_0\|\hat{v}_0\|^2\le 0.
\]
This implies that $c_{0}\le 0$. The same argument applied to
$g=(1-\delta_m^2)\hat{v}_0$, allows us to conclude
that we also have $ c_0\ge 0$, and so $c_0=0$. We now rewrite inequality \eqref{inequality} as
\begin{equation}
\label{eq:Rewrite_inner_H_1_prod}
\left \langle\hat{w}_m, \frac{g-\hat{v}_0}{\delta_m^2} \right\rangle\leq
\frac{\delta_m}2\left\| \frac{g-\hat{v}_0}{\delta_m^2}\right\|^2.
\end{equation}
For all $j=1,\ldots, n-1$, we define the function $g^j_{\theta}\in \cH_{1+s}$ by
\begin{equation*}
\begin{aligned}
g^j_{\theta}(x,y)&:= \left(x_n\cos\theta+x_j\sin\theta+\sqrt{(x_n\cos\theta+x_j\sin\theta)^2+y^2}\right)^s\\
&\qquad\times\left(x_n\cos\theta+x_j\sin\theta -s\sqrt{(x_n\cos\theta+x_j\sin\theta)^2+y^2}\right),
\end{aligned}
\end{equation*}
and we see that
\begin{equation}
\label{eq:Convergence_g_theta}
\frac{1}{\theta}\left(g^j_{\theta} - \hat v_0\right)\rightarrow (1-s^2)x_jU_0,\quad\text{as } \theta \downarrow 0,
\end{equation}
where the converge is the $H^1(B_1,|y|^a)$. We also notice that that 
\begin{equation}
\label{eq:U_o_i_j}
\langle x_i U_0, x_j U_0\rangle=0,\quad\forall\, i,j=1,\ldots, n-1,\, i\neq j.
\end{equation}
Choosing $g:=g^j_{\theta}$ with $\theta=\delta_m^2$ in inequality \eqref{eq:Rewrite_inner_H_1_prod}, letting $m$ tend to $\infty$ and using properties \eqref{eq:Convergence_g_theta} and \eqref{eq:U_o_i_j}, we obtain
$$
\langle(1-s^2) \hat{w}, x_j U_0\rangle=(1-s^2)c_j \|x_j U_0\|^2\leq 0.
$$
Hence, it follows that $c_j\le 0$, because $s\in (0,1)$. Replacing $x_j$ with $-x_j$ in the preceding
argument, we also obtain $-c_j\leq 0$. Thus, we conclude that $c_j=0$, for
all $j=1,\ldots, n-1$, which implies $\hat{w}=0$. This concludes the
proof of \eqref{weakly}.
\end{step}

\begin{step}[Proof of property (ii)] 
Finally, we claim that, along a subsequence, we have that
\begin{equation}\label{step5}
\hat{w}_m\rightarrow 0\quad\text{strongly in} \ H^1(B_1,|y|^a),\quad\text{as } m\rightarrow \infty.
\end{equation}
Because we already have the strong convergence $\hat{w}_m\to \hat{w}=0$
in $L^2(B_1,|y|^a)$, as $m$ tends to $\infty$, we are left with proving that
\begin{equation}\label{step52}
\nabla \hat{w}_m\rightarrow 0\quad\text{strongly in}\  L^2(B_1,|y|^a),\quad\text{as } m \rightarrow\infty.
\end{equation}
To this end, we pick $\eta\in C^{0,1}_0(B_1)$, such that $0\le \eta\le 1$,  and
consider $\tilde{w}_m=(1-\eta)w_m+\eta \hat{v}_0$. Clearly, we have that 
$$
\tilde{w}_m=w_m \text{ on } \partial B_1,
\quad
\tilde{w}_m\ge 0\text{ on } B_1',
\quad\text{and}\quad
\tilde{w}_m- \hat{v}_0 = (1-\eta)(w_m-\hat{v}_0).
$$
Applying inequality \eqref{theone} with this choice of the function $\tilde{w}_m$, dividing by $\delta_m^2$, and recalling that
$\hat{w}_m=\frac{w_m-\hat{v}_0}{\delta_m}$, we obtain
\begin{align*}
&(1-\kappa_m)\left(\int_{B_1}|\nabla \hat{w}_m|^2 |y|^a
- (1+s)\int_{\partial B_1}\hat{w}_m^2|y|^a
-4\int_{B_1'}\frac{w_m}{\delta_m^2} \lim_{y\downarrow 0}|y|^a \partial_y \hat{v}_0\right)
\\
&\qquad\quad \leq \int_{B_1}\left[(1-\eta)^2|\nabla \hat{w}_m|^2+\hat{w}_m^2|\nabla \eta|^2-2(1-\eta)\hat{w}_m \nabla \hat{w}_m\dotprod\nabla \eta\right]|y|^a\\
&\qquad\quad\quad - (1+s)\int_{\partial B_1}(1-\eta)^2\hat{w}_m^2|y|^a
- 4 \int_{B_1'}(1-\eta)\frac{w_m}{\delta_m^2} \lim\limits_{y\downarrow 0}|y|^a \partial_y \hat{v}_0.
\end{align*}
The preceding inequality yields
\begin{align*}
&\int_{B_1}|\nabla \hat{w}_m|^2|y|^a-4\int_{B_1'}\frac{w_m}{\delta_m^2} \lim_{y\downarrow 0}|y|^a \partial_y \hat{v}_0\\
&\qquad\le \int_{B_1}\left[(1-\eta)^2|\nabla \hat{w}_m|^2+|\nabla \eta|^2\hat{w}_m^2-2(1-\eta)\hat{w}_m \nabla \hat{w}_m \dotprod\nabla \eta\right]|y|^a\\
&\qquad\qquad - (1+s)\int_{\partial B_1}(1-\eta)^2\hat{w}_m^2|y|^a
-4\int_{B_1'}(1-\eta)\frac{w_m}{\delta_m^2} \lim\limits_{y\downarrow 0}|y|^a \partial_y \hat{v}_0\\
&\qquad\qquad+(1-\kappa_m)(1+s)\int_{\partial B_1}\hat{w}_m^2|y|^a + \kappa_m\left(\int_{B_1}|\nabla \hat{w}_m|^2|y|^a-4\int_{B_1'}\frac{w_m}{\delta_m^2}\lim_{y\downarrow 0}|y|^a \partial_y \hat{v}_0\right)\\
&\qquad=\int_{B_1}\left[(1-\eta)^2|\nabla \hat{w}_m|^2+|\nabla\eta|^2\hat{w}_m^2-2(1-\eta)\hat{w}_m \nabla \hat{w}_m \dotprod\nabla \eta\right]|y|^a\\
&\qquad\qquad +\kappa_m\left(\int_{B_1}|\nabla \hat{w}_m|^2 |y|^a
-4\int_{B_1'}\frac{w_m}{\delta_m^2}\lim\limits_{y\downarrow 0}|y|^a \partial_y \hat{v}_0 
-(1+s)\int_{\partial B_1}\hat{w}_m^2|y|^a\right)\\
&\qquad\qquad+ (1+s)\int_{\partial B_1}\left(1-(1-\eta)^2\right)\hat{w}_m^2|y|^a- 4 \int_{B_1'}(1-\eta)\frac{w_m}{\delta_m^2} \lim\limits_{y\downarrow 0}|y|^a \partial_y \hat{v}_0.
\end{align*}
From properties \eqref{eq:Uniform_H_1_norm}, \eqref{sh1} and the previous inequality, it follows that
\begin{align*}
&\int_{B_1}|\nabla \hat{w}_m|^2|y|^a
\\
&\qquad\le \int_{B_1}\left[(1-\eta)^2|\nabla \hat{w}_m|^2+|\nabla \eta|^2\hat{w}_m^2-2(1-\eta)\hat{w}_m \nabla \hat{w}_m \dotprod\nabla \eta\right]|y|^a
\\
 &\qquad\qquad+ C\kappa_m  + (1+s)\int_{\partial B_1}\hat{w}_m^2 |y|^a
+ 4\int_{B_1'}\eta\frac{w_m}{\delta_m^2} \lim_{y\downarrow 0}|y|^a \partial_y \hat{v}_0.
\end{align*}
We now make the following choice of the function $\eta$ in the preceding inequality, 
\[
\eta(x)=\begin{cases} 
1,& \text{if } |x|\le \frac{1}{2},
\\
2(1-|x|),& \text{if }\frac{1}{2}< |x| < 1,
\\
0,& \text{if } |x|\ge 1,
\end{cases}
\]
and we obtain
\begin{align*}
\int_{B_{\frac{1}{2}}}|\nabla \hat{w}_m|^2|y|^a& \le \int_{B_1}\left[|\nabla \eta|^2\hat{w}_m^2-2(1-\eta)\hat{w}_m\nabla \hat{w}_m \dotprod \nabla \eta\right]|y|^a + (1+s)\int_{\partial B_1}\hat{w}_m^2 |y|^a+ C\kappa_m \\
&\qquad+4\int_{B_1'}\eta\frac{w_m}{\delta_m^2} \lim\limits_{y\downarrow 0}|y|^a  \partial_y \hat{v}_0\\
&\le \int_{B_1}\left[|\nabla \eta|^2\hat{w}_m^2-2(1-\eta)\hat{w}_m \nabla \hat{w}_m \dotprod \nabla \eta\right] |y|^a
+(1+s)\int_{\partial B_1}\hat{w}_m^2 |y|^a
+ C\kappa_m,
\end{align*}
where in the last inequality we used inequality \eqref{eq:Limit_v_y_at_0}, and the fact that $\eta$ and $w_m$ are nonnegative functions on $B'_1$. We thus conclude that
\begin{equation}\label{eta}
\int_{B_{\frac{1}{2}}}|\nabla \hat{w}_m|^2 |y|^a \le \int_{B_1}\left[|\nabla \eta|^2\hat{w}_m^2-2(1-\eta)\hat{w}_m \nabla \hat{w}_m \dotprod \nabla \eta\right] |y|^a +(1+s)\int_{\partial B_1}\hat{w}_m^2 |y|^a
+ C\kappa_m.
\end{equation}
We now observe that, since $\hat{w}_m$ is homogeneous of degree $1+s$,
and thus $\nabla \hat{w}_m$ is homogeneous of degree $s$, we have that 
\[
\int_{B_{1}}|\nabla \hat{w}_m|^2 |y|^a= 2^{n+3} \int_{B_{\frac{1}{2}}}|\nabla \hat{w}_m|^2 |y|^a,
\]
where we recall that $a=1-2s$. Using the preceding identity in inequality \eqref{eta}, we conclude that
\begin{align*}
\int_{B_1}|\nabla \hat{w}_m|^2 |y|^a&\le 2^{n+3}\biggl(\int_{B_1}\left[|\nabla \eta|^2\hat{w}_m^2-2(1-\eta)\hat{w}_m \nabla \hat{w}_m \dotprod\nabla \eta\right] |y|^a\\
&\qquad +(1+s)\int_{\partial B_1}\hat{w}_m^2|y|^a +C\kappa_m\biggr).
\end{align*}
To complete the proof of \eqref{step5}, and consequently of Theorem~\ref{T:epi}, all we need to do at this point is to observe that, on a
subsequence, the right-hand side of the latter inequality converges to
$0$ as $m\to \infty$. This follows from the facts that $\kappa_m\to 0$,
$\|\hat{w}_m\|_{L^2(B_1, |y|^a)}\to 0$, 
$\|\hat{w}_m\|_{L^2(\partial B_1, |y|^a)}\to 0$, 
and $\|\nabla \hat{w}_m\|_{L^2(B_1, |y|^a)} \le 1$. 
\qedhere
\end{step}
\end{proof}
This completes the proof of the Claim~\ref{claim:hat_w}, and thus that of the theorem.
\end{proof}

\section{$C^{1+\gamma}$ regularity of the regular part of the free boundary}
\label{sec:Regularity_free_boundary}

In this section, we prove the main results of our article, Theorems~\ref{thm:Regularity_free_boundary} and \ref{thm:Regularity_free_boundary_with_drift}. We prove Theorem~\ref{thm:Regularity_free_boundary} using a series of intermediate results. We begin with the following analogue of \cite[Lemma~7.1]{Garofalo_Petrosyan_SmitVegaGarcia}, adapted to the framework of our article.
\begin{lem}
\label{lem:Difference_homogeneous_res}
Assume that $0\in \Gamma_{1+s}(u)$. Let $r_1\in (0,1)$, and let $w_r$ denote the $(1+s)$-homogeneous extension of the rescaling $v_r$ from $\partial B_1$ to $B_1$. For all $r\in (0,r_1)$, assume that there is a function, $\zeta_r\in H^1(B_1,|y|^a)$, such that $\zeta_r$ is nonnegative on $B'_1$, $\zeta_r=w_r$ on $\partial B_1$, and such that 
\begin{equation}
\label{eq:Inequality_W}
W(\zeta_r) \leq (1-\kappa) W(w_r),
\end{equation}
where $\kappa\in(0,1)$ is the constant appearing in Theorem~\ref{T:epi}. Then, there are positive constants, $C$ and $\beta=\beta(\kappa,n,s)\in (0,1)$, such that
\begin{equation}
\label{eq:Difference_homogeneous_res}
\int_{\partial B_1} |v_r-v_{r'}||y|^a \leq C r^{\beta},\quad 0<r'<r_1.
\end{equation}
\end{lem}

\begin{proof}
We divide the proof into several steps.
\setcounter{step}{0}
\begin{step}[Decay of $W_L(v,r)$, as $r\downarrow 0$]
\label{step:Decay_W_L}
In this step, we show that there are positive constants, $C$ and $\gamma\in (0,1)$, such that
\begin{equation}
\label{eq:Decay_W_L}
W_L(v,r) \leq C r^{\gamma},\quad\forall\, r\in (0,r_1).
\end{equation}
Our method of the proof of inequality \eqref{eq:Decay_W_L} consists in using the properties of the Weiss functional, $W_L(v,r)$, and of the boundary adjusted Weiss energy, $W(v,r)$, together with the epiperimetric inequality.

From identities \eqref{eq:Weiss_functional} and \eqref{eq:Derivative_W}, it follows that
$$
\frac{d}{dr}W_L(v,r) = -\frac{n+2}{r} W_L(v,r) + \frac{(1+s)}{r^{n+4}} F(r) + \frac{1}{r^{n+2}} I'(r)-\frac{1+s}{r^{n+3}} F'(r),
$$
and using identities \eqref{eq:I_derivative} and \eqref{eq:Derivative_F}, we have that
\begin{align*}
\frac{d}{dr}W_L(v,r) &= -\frac{n+2}{r} W_L(v,r)  + \frac{(1+s)}{r^{n+4}} F(r) +\frac{1}{r^{n+2}}\int_{\partial B_r} |\nabla v|^2|y|^a\\
 &\qquad+\frac{1}{r^{n+2}} \int_{\partial B'_r} vh -\frac{2(1+s)}{r^{n+3}} \int_{\partial B_r} v(\nabla v\dotprod\nu)|y|^a - \frac{(1+s)(n+a)}{r^{n+4}}F(r).
\end{align*}
From property \eqref{eq:Growth_v_h_on_R_n}, and denoting by $\partial_{\tau} v$ the tangential derivative of $v$ to $\partial B_r$, we obtain
\begin{align*}
\frac{d}{dr}W_L(v,r) &\geq -\frac{n+2}{r} W_L(v,r)  -Cr^{2s-2} 
+\frac{1}{r^{n+2}}\int_{\partial B_r}\left( |\nabla v\dotprod \nu|^2 + |\partial_{\tau} v|^2\right)|y|^a\\
 &\qquad -\frac{2(1+s)}{r^{n+3}} \int_{\partial B_r} v(\nabla v\dotprod\nu)|y|^a - \frac{(1+s)(n-2s)}{r^{n+4}}\int_{\partial B_r} |v|^2|y|^a,
\end{align*}
where $C$ is a positive constant. Using the definition \eqref{eq:Homogeneous_res} of the homogeneous rescalings, $\{v_r\}_{r>0}$, the preceding inequality can be rewritten in the form
\begin{equation}
\label{eq:Derivative_W_1}
\begin{aligned}
\frac{d}{dr}W_L(v,r) &\geq -\frac{n+2}{r} W_L(v,r)  -Cr^{2s-2} 
+\frac{1}{r}\int_{\partial B_1}\left( \nabla v_r\dotprod \nu - (1+s)v_r \right)^2|y|^a\\
 &\qquad  - \frac{(1+s)(n+1+s-2s)}{r}\int_{\partial B_1} |v_r|^2|y|^a
 +\frac{1}{r}\int_{\partial B_1}|\partial_{\tau} v_r|^2|y|^a.
\end{aligned}
\end{equation}
Because $w_r=v_r$ on $S_1$, we have that
\begin{equation}
\label{eq:Use_homogeneous_ext_1}
\begin{aligned}
& \int_{\partial B_1}\left(|\partial_{\tau} v_r|^2  - (1+s)(n+1-s) |v_r|^2\right)|y|^a\\
&\qquad= \int_{\partial B_1}\left(|\partial_{\tau} w_r|^2  - (1+s)(n+1-s) |w_r|^2\right)|y|^a.
\end{aligned}
\end{equation}
Using the fact that $w_r$ is $(1+s)$-homogeneous, we have that $\nabla w_r\dotprod\nu=(1+s)w_r$ on $\partial B_1$. Using in addition the fact that $|\partial_{\tau} w_r|^2=|\nabla w_r|^2-|\nabla w_r\dotprod\nu|^2$, it follows that
\begin{equation}
\label{eq:Use_homogeneous_ext_2}
\begin{aligned}
& \int_{\partial B_1}\left(|\partial_{\tau} w_r|^2  - (1+s)(n+1-s) |w_r|^2\right)|y|^a\\
&\qquad= \int_{\partial B_1}\left(|\nabla w_r|^2  - (1+s)(n+2) |w_r|^2\right)|y|^a.
\end{aligned}
\end{equation}
The $(1+s)$-homogeneity of $w_r$ also gives us
\begin{equation}
\label{eq:Use_homogeneous_ext_3}
 \int_{B_1} |\nabla w_r|^2|y|^a  =\frac{1}{n+2}\int_{\partial B_1} |\nabla w_r|^2 |y|^a.
\end{equation}
Inequalities \eqref{eq:Use_homogeneous_ext_1}--\eqref{eq:Use_homogeneous_ext_3}, and definition \eqref{eq:Boundary_Weiss_functional} of the boundary adjusted Weiss energy, yield
\begin{equation*}
\int_{\partial B_1}\left(|\partial_{\tau} v_r|^2  - (1+s)(n+1-s) |v_r|^2\right)|y|^a
= (n+2) W(w_r,1).
\end{equation*}
The preceding identity and inequality \eqref{eq:Derivative_W_1} yield
\begin{equation}
\label{eq:Derivative_W_2}
\begin{aligned}
\frac{d}{dr}W_L(v,r) &\geq \frac{n+2}{r}\left(W(w_r,1) - W_L(v,r)\right)\\  
&\qquad +\frac{1}{r}\int_{\partial B_1}\left( \nabla v_r\dotprod \nu - (1+s)v_r \right)^2|y|^a
-Cr^{2s-2}. 
\end{aligned}
\end{equation}
We next use the hypothesis that for all $r\in (0,r_1)$, there is a function, $\zeta_r\in H^1(B_1,|y|^a)$, such that $\zeta_r$ is nonnegative on $B'_1$, $\zeta_r=w_r$ on $\partial B_1$, and such that inequality \eqref{eq:Inequality_W} holds. Without loss of generality, we may assume that $\zeta_r$ is a minimizer of $W(\cdot,1)$ in the class of functions
$$
\cC:=\{\zeta\in H^1(B_1,|y|^a)\mid \zeta=v_r=w_r \text{ on } \partial B_1,\text{ and } \zeta \geq 0 \text{ on } B'_1\}.
$$
This is equivalent to minimizing the energy $\int_{B_1} |\nabla \zeta|^2|y|^a$ among the class of functions $\cC$, and so a standard calculus of variations argument implies that $\zeta_r$ is a $L_a$-superharmonic function, that is
\begin{equation}
\label{eq:Superharmonic_1}
\int_{B_1} \nabla \zeta_r\dotprod\nabla\varphi|y|^a \geq 0,
\end{equation}
for all nonnegative test functions, $\varphi\in H^1(B_1,|y|^a)$, with $\supp(\varphi)\subseteq B_1$, and also,
$$
\int_{B_1} \nabla \zeta_r\dotprod\nabla\varphi|y|^a = 0,
$$
for all test functions, $\varphi\in H^1(B_1,|y|^a)$, such that $\supp(\varphi)\subseteq B_1\setminus(B'_1\cap\{\zeta_r=0\})$. The preceding identity implies that 
\begin{equation}
\label{eq:Superharmonic_2}
L_a\zeta_r = 0\quad\text{a.e.\ on } B_1\setminus(B'_1\cap\{\zeta_r>0\}).
\end{equation}
Given a nonnegative test function, $\varphi\in H^1(B_1,|y|^a)$, with $\supp(\varphi)\subseteq B_1$, we have that
\begin{align*}
\int_{B_1} \nabla \zeta_r^+\dotprod\nabla\varphi|y|^a &= \int_{\partial\{\zeta_r>0\}\cap B_1} \nabla\zeta_r\dotprod\nu\varphi|y|^a-\int_{\{\zeta_r>0\}\cap B_1} L_a \zeta_r \varphi.
\end{align*}
The preceding identity together with property \eqref{eq:Superharmonic_2}, and the fact that the normal derivative $\nabla\zeta_r\dotprod\nu \leq 0$ on $\partial\{\zeta_r>0\}\cap B_1$, implies that
$$
\int_{B_1} \nabla \zeta_r^+\dotprod\nabla\varphi|y|^a \leq 0,
$$
and so, $\zeta_r^+$ is a $L_a$-subharmonic function. Inequality \eqref{eq:Superharmonic_1} gives us that
$$
\int_{B_1} \nabla \zeta_r^+\dotprod\nabla\varphi|y|^a \geq \int_{B_1} \nabla \zeta_r^-\dotprod\nabla\varphi|y|^a,
$$
for all nonnegative test functions, $\varphi\in H^1(B_1,|y|^a)$, with $\supp(\varphi)\subseteq B_1$, and so $\zeta_r^-$ is also a $L_a$-subharmonic function. We now let
$$
\hat \zeta_r(x,y) := r^{1+s}\zeta_r((x,y)/r),\quad\forall\, (x,y)\in B_r,
$$
and we see that $\hat \zeta_r = v$ on $\partial B_r$, and using definition \eqref{eq:Boundary_Weiss_functional} of the boundary adjusted Weiss energy, we have 
\begin{equation}
\label{eq:W_with_tilde_zeta}
W(\zeta_r,1) = \frac{1}{r^{n+2}} \int_{B_r} |\nabla\hat\zeta_r|^2|y|^a -\frac{1+s}{r^{n+3}} \int_{\partial B_r} |v|^2|y|^a.
\end{equation}
Because $v$ verifies conditions \eqref{eq:Properties_v_1}--\eqref{eq:Equality_L_a} on $B_r$, instead of $\RR^{n+1}$, we see that $v$ is a minimizer of the energy
$$
\int_{B_r} |\nabla\varphi|^2|y|^a + \int_{B'_r} \varphi h,
$$
in the class of functions $\{\varphi\in H^1(B_r,|y|^a)\mid \varphi=v \text{ on } \partial B_r, \varphi \geq 0 \text{ on } B'_r\}$. In particular, this implies 
$$
\int_{B_r} |\nabla\hat\zeta_r|^2|y|^a + \int_{B'_r} \hat\zeta_r h \geq \int_{B_r} |\nabla v|^2|y|^a + \int_{B'_r} v h,
$$
Because the functions $\zeta_r^{\pm}$ are $L_a$-subharmonic on $B_1$, we have that $\hat \zeta_r$ is also $L_a$-subharmonic on $B_r$, and the weak maximum principle \cite[Theorem~2.2.2]{Fabes_Kenig_Serapioni_1982a} implies
$$
\sup_{B_r} |\hat \zeta^{\pm}_r| \leq \sup_{\partial B_r} |v|.
$$
From Lemma~\ref{lem:Growth_v_balls}, it follows that there exists $C>0$ such that $|v(x,y)|\leq Cr^{1+s}$ on $B_r$, and so we have
$$
|\hat\zeta_r(x,y)| \leq C r^{1+s},\quad\forall\, (x,y)\in B_r,\quad\forall\, r\in (0,1).
$$
Combining the preceding three inequalities with \eqref{eq:Growth_v_h_on_R_n}, we find
$$
\int_{B_r} |\nabla\hat\zeta_r|^2|y|^a  \geq \int_{B_r} |\nabla v|^2|y|^a - Cr^{n+1+2s}.
$$
Using the preceding inequality with \eqref{eq:W_with_tilde_zeta}, it follows that
\begin{align*}
W(\zeta_r,1) \geq \frac{1}{r^{n+2}} \int_{B_r} |\nabla v|^2|y|^a -\frac{1+s}{r^{n+3}} \int_{\partial B_r} |v|^2|y|^a -C r^{2s-1},
\end{align*}
and so, definition \eqref{eq:Weiss_functional} of the Weiss functional gives 
\begin{align*}
W(\zeta_r,1) \geq W_L(v,r) -C r^{2s-1}.
\end{align*}
Hypothesis \eqref{eq:Inequality_W} and the preceding inequality imply 
\begin{align}
W(w_r,1)-W_L(v,r) &\geq\frac{1}{1-\kappa} W(\zeta_r,1) - W_L(v,r)\notag\\
\label{eq:Diff_Weiss_functs}
&\geq \frac{\kappa}{1-\kappa} W_L(v,r) - Cr^{2s-1}.
\end{align}
We now obtain from inequality \eqref{eq:Derivative_W_2} 
\begin{align*}
\frac{d}{dr} W_L(v,r) &\geq \frac{n+2}{r}\frac{\kappa}{1-\kappa} W_L(v,r) - Cr^{2s-2}.
\end{align*}
This estimate implies that for any $\gamma>0$ one has
\begin{align*}
\frac{d}{dr} \left(W_L(v,r)r^{-\gamma}\right) &=  \frac{d}{dr} W_L(v,r)r^{-\gamma} - \gamma W_L(v,r) r^{-\gamma-1}\\
&\geq \left(\frac{(n+2)\kappa}{1-\kappa} -\gamma\right) W_L(v,r) r^{-\gamma-1} - Cr^{2s-2-\gamma}.
\end{align*}
Choosing $\gamma < (n+2)\kappa/(1-\kappa)$, and using Lemmas~\ref{lem:Monotonicity_Weiss} and \ref{lem:Weiss_0_regular_points}, it follows that there exists $C>0$ such that
\begin{align*}
\frac{d}{dr} \left(W_L(v,r)r^{-\gamma}\right) 
&\geq -C\left(\frac{(n+2)\kappa}{1-\kappa} -\gamma\right) r^{2s-2-\gamma} - Cr^{2s-2-\gamma}=-Cr^{2s-2-\gamma}.
\end{align*}
Integrating the preceding inequality from $r$ to $r_1$, with $r>0$, we infer
\begin{align*}
W_L(v,r_1) r_1^{-\gamma} - W_L(v,r) r^{-\gamma} \geq -Cr_1^{2s-1-\gamma} + Cr^{2s-1-\gamma},
\end{align*}
from which inequality \eqref{eq:Decay_W_L} now follows. This completes the proof of Step~\ref{step:Decay_W_L}.
\end{step}

\begin{step}
\label{step:Estimate_homogeneity}
We now show that there exists $C>0$ such that for all $r\in (0,r_1)$ one has
\begin{equation}
\label{eq:Estimate_homogeneity}
\frac{1}{r}\int_{\partial B_1}\left(\nabla v_r\dotprod\nu-(1+s)v_r\right)^2|y|^a
\leq \frac{d}{dr} W_L(v,r) + Cr^{2s-2}.
\end{equation}
From inequality \eqref{eq:Derivative_W_2}, it follows that
\begin{equation*}
\frac{1}{r}\int_{\partial B_1}\left(\nabla v_r\dotprod\nu-(1+s)v_r\right)^2|y|^a
\leq \frac{d}{dr} W_L(v,r) -\frac{n+2}{r}\left(W(w_r,1)-W_L(v,r)\right)+ Cr^{2s-2}.
\end{equation*}
Furthermore, inequality \eqref{eq:Diff_Weiss_functs} gives
\begin{equation*}
\frac{1}{r}\int_{\partial B_1}\left(\nabla v_r\dotprod\nu-(1+s)v_r\right)^2|y|^a
\leq \frac{d}{dr} W_L(v,r) -\frac{n+2}{r}\frac{\kappa}{1-\kappa} W_L(v,r)+  Cr^{2s-2}.
\end{equation*}
Lemmas~\ref{lem:Monotonicity_Weiss} and \ref{lem:Weiss_0_regular_points} imply that $W_L(v,r)\geq -Cr^{2s-1}$. Combining this with the preceding inequality yields \eqref{eq:Estimate_homogeneity}. This concludes the proof of Step~\ref{step:Estimate_homogeneity}.
\end{step}

\begin{step}[Proof of estimate \eqref{eq:Difference_homogeneous_res}]
\label{step:Difference_homogeneous_res}
Let $0<r'<r<r_1$, and denote $g(r) = v_r$. Direct calculations give
\begin{align*}
\int_{\partial B_1} |v_r-v_{r'}||y|^a &= \int_{\partial B_1}\left|\int_{r'}^r g'(t)\, dt\right||y|^a\\
&= \int_{\partial B_1}\left|\int_{r'}^r \left(\frac{1}{t^{1+s}}\nabla v(t(x,y))\dotprod (x,y) -\frac{1+s}{t}\frac{v(t(x,y))}{t^{1+s}}\right)\, dt\right||y|^a\\
&\leq \int_{r'}^r \frac{1}{t}\int_{\partial B_1} \left|\nabla v_t \dotprod \nu - (1+s)v_t \right||y|^a \, dt.
\end{align*}
H\"older's inequality implies
\begin{align*}
\int_{\partial B_1} |v_r-v_{r'}||y|^a 
&\leq C\int_{r'}^r \frac{1}{\sqrt{t}}\left(\frac{1}{t}\int_{\partial B_1} \left|\nabla v_t \dotprod \nu - (1+s)v_t \right|^2|y|^a\right)^{1/2} \, dt,
\end{align*}
where $C=C(n,s)>0$. Using inequality \eqref{eq:Estimate_homogeneity}, we conclude that
\begin{align*}
\int_{\partial B_1} |v_r-v_{r'}||y|^a 
&\leq C\int_{r'}^r \frac{1}{\sqrt{t}}\left(\frac{d}{dt}W_L(v,t) + Ct^{2s-2}\right)^{1/2} \, dt.
\end{align*}
Applying H\"older's inequality again to the right-hand side of the latter inequality gives
\begin{align*}
\int_{\partial B_1} |v_r-v_{r'}||y|^a 
&\leq C\left(\int_{r'}^r \frac{1}{t}\, dt\right)^{1/2} 
\left(\int_{r'}^r\left(\frac{d}{dt}W_L(v,t) + Ct^{2s-2}\right) \, dt\right)^{1/2}\\
&= C  \left( \ln r/r'\right)^{1/2} \left(W_L(v,r) - W_L(v,r') + Cr^{2s-1} - C(r')^{2s-1}\right)^{1/2}.
\end{align*}
The assumption $s>1/2$, estimate \eqref{eq:Decay_W_L}, and the fact that $W_L(v,r) \geq -Cr^{2s-1}$, from Lemmas~\ref{lem:Monotonicity_Weiss} and \ref{lem:Weiss_0_regular_points}, imply 
\begin{align*}
\int_{\partial B_1} |v_r-v_{r'}||y|^a 
&= C  \left( \ln r/r'\right)^{1/2} \left(r^{\gamma} + Cr^{2s-1}\right)^{1/2}.
\end{align*}
Letting $\beta:=\gamma\wedge (2s-1)$, we can now repeat the dyadic argument in \cite[Estimate (7.2) on p. 29]{Garofalo_Petrosyan_SmitVegaGarcia} to finally obtain \eqref{eq:Difference_homogeneous_res}.
\end{step}
This completes the proof.
\end{proof}

\begin{prop}
\label{prop:Decay_homogeneous_res}
Let $x_0\in\Gamma_{1+s}(u)$. Then, there exist constants $C, \eta, r_0> 0$, and $\beta=\beta(\kappa,n,s)\in (0,1)$, such that $B'_{\eta}(x_0)\cap\Gamma(u)\subseteq \Gamma_{1+s}(u)$, and for all $x\in B'_{\eta}(x_0)\cap\Gamma(u)$ and all $r\in (0,r_0)$, we have that
\begin{equation}
\label{eq:Decay_homogeneous_res}
\int_{\partial B_1} |v_{x,r}-v_{x,0}| |y|^a \leq Cr^{\beta},
\end{equation}
where $v_{x,0}$ is any limit of a convergent sequence of homogeneous rescalings, $\{v_{x,r_k}\}_{k\in\NN}$, with $r_k\downarrow 0$. In particular, the blowup limit $v_{x,0}$ is unique.
\end{prop}

\begin{proof}
The method of the proof is exactly the same as that of \cite[Proposition~7.2]{Garofalo_Petrosyan_SmitVegaGarcia}, with the observations that we choose the positive constant $r_0$ as in Lemma~\ref{lem:Convergence_homogeneous_functions}, we set $r_1=r_0$ in Lemma~\ref{lem:Difference_homogeneous_res}, and we replace the application of \cite[Lemma~3.3]{Garofalo_Petrosyan_SmitVegaGarcia} with that of Lemma~\ref{lem:Open_set_regular_points}, of \cite[Lemma~3.4]{Garofalo_Petrosyan_SmitVegaGarcia} with that of Lemma~\ref{lem:Convergence_homogeneous_functions}, and that of \cite[Lemma~7.1]{Garofalo_Petrosyan_SmitVegaGarcia} with that of Lemma~\ref{lem:Difference_homogeneous_res}. We omit the detailed proof for brevity. 
\end{proof}

We next have the analogue of \cite[Proposition~7.3]{Garofalo_Petrosyan_SmitVegaGarcia} in which we establish that the blowup limits are nontrivial.

\begin{prop}
\label{prop:Nontrivial_blowup_limits}
Assume that the hypotheses of Proposition~\ref{prop:Decay_homogeneous_res} hold. Then, for every $x\in B'_{\eta}(x_0)\cap\Gamma(u)$ the unique blowup limit $v_{x,0}$ is nonzero, where $\eta$ is the positive constant appearing in the statement of Proposition~\ref{prop:Decay_homogeneous_res}.
\end{prop} 

\begin{proof}
Assume by contradiction that $v_{x,0} \equiv 0$. Proposition~\ref{prop:Decay_homogeneous_res} implies that there exist $C, r_0>0$ such that
$$
\int_{\partial B_1} |v_{x,r}||y|^a \leq Cr^{\beta},\quad\forall\, r\in (0,r_0),
$$
and definitions \eqref{eq:Rescaling} and \eqref{eq:Homogeneous_res} give 
\begin{equation}
\label{eq:Trivial_blowup_limit}
\int_{\partial B_1} |\tilde v_{x,r}||y|^a \leq C\frac{r^{1+s+\beta}}{d_{x,r}},\quad\forall\, r\in (0,r_0).
\end{equation}
Proposition~\ref{prop:Monotonicity_formula} and the fact that $x\in\Gamma_{1+s}(u)$ (see Definition~\ref{defn:Regular_points}) imply that $F_{x,r} >r^{n+a+2(1+p)}$, for all $p \in (s,2s-1/2)$. The preceding inequality together with identity \eqref{eq:d_r} imply that $d_{x,r}> r^{1+p}$. We see that we can choose $p\in (s,2s-1/2)$, such that $\beta+s-p>0$, and letting $r$ tend to 0 in \eqref{eq:Trivial_blowup_limit} gives
\begin{equation*}
\lim_{r\downarrow 0}\int_{\partial B_1} |\tilde v_{x,r}||y|^a =0.
\end{equation*}
This contradicts property \eqref{eq:Convergence_homogeneous_functions}, which shows that the limit above is non-trivial. We thus conclude that the unique blowup limit $v_{x,0}$ is nontrivial.
\end{proof}

\begin{prop}
\label{prop:Difference_blowup_lim_different_points}
Assume that the hypotheses of Proposition~\ref{prop:Decay_homogeneous_res} hold. Then there are positive constants, $C$ and $\gamma=\gamma(\kappa,n,s)\in (0,1)$, such that
\begin{equation}
\label{eq:Difference_blowup_different_points}
\int_{\partial B'_1} |v_{x',0} - v_{x'',0}|  \leq C|x'-x''|^{\gamma},\quad\forall\, x', x'' \in B'_{\eta}(x_0)\cap \Gamma(u),
\end{equation}
where $\eta$ is the positive constant appearing in the statement of Proposition~\ref{prop:Decay_homogeneous_res}.
\end{prop}

\begin{proof}
Since $v_{x',0} - v_{x'',0}$ is a $1+s$ homogeneous function, proving inequality \eqref{eq:Difference_blowup_different_points} is equivalent to establishing the following one
\begin{equation}\label{eq:2Difference_blowup_different_points}
\int_{B'_1} |v_{x',0} - v_{x'',0}|  \leq C|x'-x''|^{\gamma},\quad\forall\, x', x'' \in B'_{\eta}(x_0)\cap \Gamma(u).
\end{equation}
Let $\eta$ and $r_0$ be the positive constants appearing in the conclusion of Proposition~\ref{prop:Decay_homogeneous_res}. Consider $r\in (0,r_0)$ and $x', x'' \in B'_{\eta}(x_0)\cap\Gamma(u)$. Property \eqref{eq:Decay_homogeneous_res} implies for all $ x', x'' \in B'_{\eta}(x_0)\cap \Gamma(u)$ and every $r\in (0,r_0)$
\begin{equation}
\label{eq:Difference_blowup_different_points_1}
\int_{\partial B_1} |v_{x',0} - v_{x'',0}| |y|^a \leq Cr^{\beta} + \int_{\partial B_1} |v_{x',r} - v_{x'',r}| |y|^a.
\end{equation}
From the mean value theorem and definition \eqref{eq:Homogeneous_res} of the homogeneous rescalings we infer
\begin{align*}
v_{x',r}(x,y) - v_{x'',r}(x,y) =\frac{1}{r^{1+s}} \int_0^1 \nabla_x v(tx'+(1-t)x''+rx,ry)\dotprod(x'-x'')\, dt,\quad\forall\, (x,y)\in \bar B_1.
\end{align*}
If we use the estimate (see the proof of Lemma~\ref{lem:Growth_v_balls})
$$
|\nabla_x v(tx'+(1-t)x''+rx,ry)| \leq
C\left(|x'-x''|^s+r^s\right),\quad (x,y)\in \bar B_1,
$$
we conclude that
$$
\int_{\partial B_1} |v_{x',r} - v_{x'',r}| |y|^a \leq C\left(\left(\frac{|x'-x''|}{r}\right)^{1+s} + \frac{|x'-x''|}{r}\right).
$$
We now let $r:=|x'-x''|^{\sigma}$, where $\sigma\in (0,1)$ is arbitrarily fixed. Then, inequality \eqref{eq:Difference_blowup_different_points_1} becomes
\begin{equation*}
\int_{\partial B_1} |v_{x',0} - v_{x'',0}| |y|^a \leq C\left(|x'-x''|^{\sigma\beta} + |x'-x''|^{1-\sigma}\right).
\end{equation*}
We now choose $2\gamma:=\sigma\beta\wedge(1-\sigma)$. The latter inequality and the $1+s$-homogeneity of $v_{x',0} - v_{x'',0}$ then give
\begin{equation}
\label{eq:Difference_blowup_different_points_2}
\int_{B_1} |v_{x',0} - v_{x'',0}| |y|^a \leq C|x'-x''|^{2\gamma}.
\end{equation}
The inequality \eqref{eq:Difference_blowup_different_points_2}, combined with the uniform sup estimate of  $|v_{x',0} - v_{x'',0}|$ (see Lemma~\ref{lem:Growth_v_free_boundary_point}), allows to conclude
\begin{equation}
\label{eq:Difference_blowup_different_points_2squared}
\int_{B_1} |v_{x',0} - v_{x'',0}|^2 |y|^a \leq C|x'-x''|^{2\gamma}.
\end{equation}
To obtain estimate \eqref{eq:Difference_blowup_different_points} from \eqref{eq:Difference_blowup_different_points_2squared}, we next use the 
the trace theorem in \cite[Theorem~2.8]{Nekvinda_1993}, which gives 
\begin{align}\label{good}
\int_{B'_1} |v_{x',0}- v_{x'',0}|^2 & \leq C \|v_{x',0}- v_{x'',0}\|^2_{H^1(B_1,|y|^a)}
\\
& = \int_{B_1} |v_{x',0} - v_{x'',0}|^2 |y|^a + \int_{B_1} |\nabla (v_{x',0}- v_{x'',0})|^2 |y|^a.
\notag
\end{align}
To control the second term in the right-hand side of the latter inequality we now exploits the fact that the blowup limits verify the conditions \eqref{eq:Properties_v_1}--\eqref{eq:Equality_L_a}, with $h_{x_0}$ replaced by 0, and $x_0$ replaced by $x'$ and $x''$. These conditions imply that 
\begin{align*}
&\int_{B_1} L_a v_{x,0} v_{x,0} = 0,\quad\text{where } x=x' \text{ or } x=x'',\\
&\int_{B_1} L_a v_{x',0} v_{x'',0} \leq 0,\\
&\int_{B_1} L_a v_{x'',0} v_{x',0} \leq 0.
\end{align*}
From these equations we infer
$$
\int_{B_1} L_a(v_{x',0}-v_{x'',0}) (v_{x',0}-v_{x'',0}) \geq 0.
$$
Integrating by parts in the preceding inequality yields
$$
\int_{B_1}|\nabla (v_{x',0}- v_{x'',0})|^2|y|^a \leq \int_{\partial B_1} \left(\nabla (v_{x',0}- v_{x'',0})\dotprod\nu \right) (v_{x',0}- v_{x'',0}) |y|^a.
$$
Using the fact that $v_{x',0}$ and $v_{x'',0}$ are $(1+s)$-homogeneous functions, from the preceding inequality we find
$$
\int_{B_1}|\nabla (v_{x',0}- v_{x'',0})|^2|y|^a \leq (1+s)\int_{\partial B_1} |v_{x',0}- v_{x'',0}|^2 |y|^a \le C \int_{B_1} |v_{x',0}- v_{x'',0}|^2 |y|^a,
$$
where in the second inequality we have again used the homogeneity of $v_{x',0}- v_{x'',0}$. Substituting this information in \eqref{good} we conclude
\begin{equation}\label{verygood}
\int_{B'_1} |v_{x',0}- v_{x'',0}|^2  \leq C \int_{B_1} |v_{x',0} - v_{x'',0}|^2 |y|^a.
\end{equation}
Combining \eqref{verygood} with \eqref{eq:Difference_blowup_different_points_2squared}, we finally obtain
\[
\int_{B'_1} |v_{x',0}- v_{x'',0}|^2  \leq C|x'-x''|^{2\gamma}.
\]
The sought for conclusion \eqref{eq:Difference_blowup_different_points} now immediately follows from this latter estimate and the uniform estimates of $v_{x',0}- v_{x'',0}$ in sup norm already invoked above.
\end{proof}

\begin{lem}[Blowup limits are in $\cH_{1+s}$]
\label{lem:Blowup_limits_homogeneous}
Assume that the hypotheses of Proposition~\ref{prop:Decay_homogeneous_res} hold. Then, for all $x\in B'_{\eta}(x_0)\cap\Gamma(u)$ the unique blowup limit $v_{x,0}$ belongs to $\cH_{1+s}$.
\end{lem} 

\begin{proof}
By Lemma~\ref{lem:Uniform_Schauder_estimates}, it follows that the sequence of Almgren rescalings $\{\tilde v_{x,r}\}_{r>0}$ contains a convergent subsequence to a function $\tilde v_{x}$.  Lemma~\ref{lem:Convergence_homogeneous_functions} implies that the limit $\tilde v_x$ belongs to $\cH_{1+s}$. From the definition \eqref{eq:Homogeneous_res} of the homogeneous rescalings, and that of the quantity $d_{x,r}$ in \eqref{eq:d_r}, it follows that
$$
\int_{\partial B_1} |v_{x,r}|^2 |y|^a = \frac{d^2_{x,r}}{r^{2(1+s)}},\quad\forall\, r>0.
$$
H\"older's inequality together with property \eqref{eq:Decay_homogeneous_res} give 
$$
\lim_{r\downarrow 0} \int_{\partial B_1} |v_{x,r}|^2 |y|^a = \int_{\partial B_1} |v_{x,0}|^2|y|^a,
$$
and the right-hand side is positive by Proposition~\ref{prop:Nontrivial_blowup_limits}. The preceding two properties together with the definitions of the homogeneous rescalings in \eqref{eq:Homogeneous_res}, and of the Almgren-type rescalings in \eqref{eq:Rescaling}, imply
$$
v_{x,r} = \frac{d_{x,r}}{r^{1+s}}\,\tilde v_{x,r} \rightarrow \left(\int_{\partial B_1} |v_{x,0}|^2|y|^a\right)^{1/2}\tilde v_x,\quad\text{as } r\downarrow 0.
$$
Because the function $\tilde v_x$ belongs to $\cH_{1+s}$, this concludes the proof of Lemma~\ref{lem:Blowup_limits_homogeneous}. 
\end{proof}

Summarizing, we have proved that, with $\eta$ as in Proposition~\ref{prop:Decay_homogeneous_res}, then for every $x'\in B'_{\eta}(x_0)\cap \Gamma(u)$ there exist a constant, $a_{x'}>0$, and a vector $e_{x'}\in\partial B'_1$, such that
$$
v_{x',0}(x,y) = a_{x'}\left(x\dotprod e_{x'} + \sqrt{(x\dotprod e_{x'})^2+y^2}\right)^s
\left(x\dotprod e_{x'} - s\sqrt{(x\dotprod e_{x'})^2+y^2}\right),\quad\forall\, (x,y)\in\RR^{n+1}.
$$
\begin{lem}
\label{lem:Characteristics_blowup_limits}
Assume that the hypotheses of Proposition~\ref{prop:Decay_homogeneous_res} hold. Then, there exist constants $C>0$ and $\gamma=\gamma(\kappa,n,s)\in (0,1)$ such that for all $x',x''\in B'_{\eta}(x_0)\cap \Gamma(u)$ one has
\begin{align}
\label{eq:Characteristics_a}
|a_{x'}-a_{x''}| &\leq C|x'-x''|^{\gamma},\\
\label{eq:Characteristics_e}
|e_{x'}-e_{x''}| &\leq C|x'-x''|^{\gamma}.
\end{align}
\end{lem}

\begin{proof}
Similarly to the proof of \cite[Lemma~7.5]{Garofalo_Petrosyan_SmitVegaGarcia}, inequality \eqref{eq:Characteristics_a} follows from the fact that there exists $C=C(n,s)>0$ such that
$$
\|v_{x,0}\|_{L^1(\partial B'_1)} = Ca_{x},\quad\forall\, x\in B'_{\eta}(x_0)\cap\Gamma(u).
$$
Thus, inequality \eqref{eq:Difference_blowup_different_points}, together with the triangle inequality, implies \eqref{eq:Characteristics_a}.

To prove inequality \eqref{eq:Characteristics_e}, because $a_{x_0}$ is a positive constant, by Proposition~\ref{prop:Nontrivial_blowup_limits}, we may assume without loss of generality that the positive constant $\eta$ is small enough so that $a_x\geq a_{x_0}/2$, for all $x\in B'_{\eta}(x_0)\cap\Gamma(u)$. Inequalities \eqref{eq:Characteristics_a} and \eqref{eq:Difference_blowup_different_points} give 
$$
\int_{\partial B'_1} \left|\frac{1}{a_{x'}} v_{x',0}-\frac{1}{a_{x''}} v_{x'',0}\right| \leq C|x'-x''|^{\gamma},\quad\forall\, x',x''\in B'_{\eta}(x_0)\cap\Gamma(u).
$$
Using definition \eqref{eq:H_1_plus_s} of the class of functions $\cH_{1+s}$ in the preceding inequality, we obtain that
$$
\int_{\partial B'_1} \left| x\dotprod e_{x'}\mathbf{1}_{\{x\dotprod e_{x'}>0\}} - x\dotprod e_{x''}\mathbf{1}_{\{x\dotprod e_{x''}>0\}}\right|\, dx
 \leq C|x'-x''|^{\gamma},
$$
which immediately implies \eqref{eq:Characteristics_e}. This completes the proof.
\end{proof}

\begin{proof}[Proof of Theorem~\ref{thm:Regularity_free_boundary}]
The method of the proof is similar to that of \cite[Theorem~7.6]{Garofalo_Petrosyan_SmitVegaGarcia}, but we include it for clarity and completeness. We divide the proof into several steps. 

\setcounter{step}{0}
\begin{step}
\label{step:Convergence_homogeneous_res_x_r}
Let $\eta$ be the positive constant in Proposition~\ref{prop:Decay_homogeneous_res}. Our goal is to prove that for all $\eps>0$, there exists $r_{\eps}>0$ such that
\begin{equation}
\label{eq:Convergence_homogeneous_res_x_r}
\|v_{x,r}-v_{x,0}\|_{C^1_a(\bar B_1^+)} < \eps,\quad\forall\, x\in B'_{\eta/2}(x_0)\cap\Gamma(u),\quad\forall\, r\in (0,r_{\eps}).
\end{equation}
Assuming by contradiction that \eqref{eq:Convergence_homogeneous_res_x_r} does not hold, it follows that there is $\eps_0>0$, and there is a sequence $\{r_k\}_{k\in\NN}$ convergent to $0$, and a sequence of points, $\{x_k\}_{k\in\NN}\subseteq B'_{\eta/2}(x_0)\cap\Gamma(u)$, such that
\begin{equation}
\label{eq:Assump_contrad}
\|v_{x_k,r_k}-v_{x_k,0}\|_{C^1_a(\bar B_1^+)} \geq \eps_0,\quad\forall\,k \in \NN.
\end{equation}
We can assume without loss of generality that the sequence of points $\{x_k\}_{k\in\NN}$ converges to $\bar x\in \overline{B'_{\eta/2}(x_0)}\cap\Gamma(u)$, and using the uniform Schauder estimate \eqref{eq:Schauder_estimates_homogeneous_res}, we can assume without loss of generality that the sequence $\{v_{x_k, r_k}\}_{k\in \NN}$ converges in $C^{1+\alpha'}_a(\bar B_1^+)$, for all $\alpha'\in (0,\alpha)$, to a function $w\in C^{1+\alpha}_a(\bar B_1^+)$.

We next prove that $w=v_{\bar x,0}$. Integrating inequality \eqref{eq:Decay_homogeneous_res}, and using definition \eqref{eq:Homogeneous_res} of the homogeneous rescalings, we have that
$$
\|v_{x,r}-v_{x,0}\|_{L^1(B_1,|y|^a)} \leq Cr^{\beta},\quad\forall\, x\in B'_{\eta}(x_0),\quad\forall\, r\in (0,r_0),
$$ 
where $r_0$ is the positive constant in Proposition~\ref{prop:Decay_homogeneous_res}. Lemmas~\ref{lem:Blowup_limits_homogeneous} and \ref{lem:Characteristics_blowup_limits} imply that $\{v_{x_k,0}\}_{k\in\NN}$ converges to $v_{\bar x,0}$ in $H^1(B_1,|y|^a)$, as $k\rightarrow\infty$. Thus, we obtain that indeed $w=v_{\bar x,0}$. Because the sequences $\{v_{x_k,r_k}\}_{k\in\NN}$ and $\{v_{x_k,0}\}_{k\in\NN}$ both converge to $v_{\bar x,0}$ in $L^1(B_1,|y|^a)$, this contradicts our assumption \eqref{eq:Assump_contrad}. 
\end{step}

\begin{step}
\label{step:Inclusion_cone}
For a given $\eps>0$ and a unit vector $e\in\RR^{n}$, define the cone 
$$
\mathcal{C}_{\eps}(e)=\{x\in\RR^{n}\mid x\dotprod e \geq \eps |x|\}.
$$
We then claim that, there is a positive constant, $r_{\eps}$, such that for any $x\in B'_{\eta/2}(x_0)\cap \Gamma(u)$, we have
\begin{equation}
\label{eq:Inclusion_cone}
\mathcal{C}_{\eps}(e_x)\cap B'_{r_{\eps}}\subseteq \{v_x(\cdot, 0)>0\}.
\end{equation}
Indeed, consider a cutout from the sphere $\partial B'_{1/2}$ by the cone $\mathcal{C}_{\eps}(e)$,
$$
K_{\eps}(e)= \mathcal{C}_{\eps}(e)\cap \partial B'_{1/2},
$$
and note that 
$$
K_{\eps}(e_x)\Subset\{v_{x, 0}(\cdot, 0)>0\}\cap B_1',\quad\text{and}\quad
v_{x, 0}(\cdot, 0)\geq a_x c_{\eps}\quad\text{on }K_{\eps}(e_x),
$$
for some positive universal constant $c_{\eps}$. Invoking Proposition~\ref{prop:Nontrivial_blowup_limits}, without loss of generality  we may assume that 
$a_x\geq a_{x_0}/2$, for all $x\in B'_{\eta_0}(x_0)\cap \Gamma(u)$. Applying inequality \eqref{eq:Convergence_homogeneous_res_x_r}, we can thus find a positive constant $r_{\eps}$, such that
$$
v_{x, r}(\cdot, 0)> 0\quad\text{on } K_{\eps}(e_x),\quad\forall\,r \in (0,r_{\eps}).
$$
Scaling back by $r$, we have
$$
v_x(\cdot, 0)> 0\quad\text{on }r K_{\eps}(e_x):=\mathcal{C}_{\eps}(e_x)\cap \partial B'_{r/2},\quad\forall\, r\in (0,r_{\eps}).
$$
Taking the union over all $r<r_{\eps}$, we obtain that the inclusion \eqref{eq:Inclusion_cone} holds. 
\end{step}

\begin{step}
\label{step:Inclusion_opposite_cone}
We next claim that for any $\eps>0$,  there exists a positive constant, $r_{\eps}$, such that for any $x\in B'_{\eta/2}(x_0)\cap \Gamma(u)$, we have
\begin{equation}
\label{eq:Inclusion_opposite_cone}
- \mathcal{C}_{\eps}(e_x)\cap B'_{r_{\eps}}\subseteq \{v_x(\cdot,0)=0\}.
\end{equation}
To prove \eqref{eq:Inclusion_opposite_cone} we note that $- K_{\eps}(e_x)\Subset\{v_{x, 0}(\cdot, 0)=0\}\cap B_1'$, and we also have that
\begin{align*}
\lim_{y\to 0+}|y|^a\partial_y v_{x, 0}(\cdot, y)\leq -a_x c_{\eps}<-(a_{x_0}/2)c_{\eps}
\quad\text{on }- K_{\eps}(e_x),
\end{align*}
for a positive universal constant $c_{\eps}$. Then, inequality \eqref{eq:Convergence_homogeneous_res_x_r} implies that there is a positive constant, $r_{\eps}$, such that
\begin{equation}
\label{eq:Pos_normal_derivative}
\lim_{y\to 0+}|y|^a\partial_y v_{x,r}(\cdot,y) < -(a_{x_0}/4)c_{\eps}\quad\text{on }-K_{\eps}(e_x),\quad\forall\, r\in (0,r_{\eps}).
\end{equation}
We claim that this implies that
$$
v_{x, r}(\cdot,0) =0\quad\text{on }-K_{\eps}(e_x),\quad\forall\, r\in (0, r_{\eps}).
$$
Indeed, from identity \eqref{eq:Equality_L_a}, and inequality \eqref{eq:Growth_v_h_on_R_n}, it follows that
\begin{align*}
\lim_{y\to 0+}|y|^a|\partial_y v_{x,r}(z,y)| &=r^{2s}\left|\frac{h_x(rz)}{r^{1+s}}\right|\\
&\leq C r^{2s-1},
\end{align*}
for all $z\in \{v_{x,r}(\cdot,0)>0\}$. If there were $z\in \{v_{x,r}(\cdot, 0)>0\}\cap -K_{\eps}(e_x)$, then when $r$ is small enough the previous inequality would give us a contradiction with \eqref{eq:Pos_normal_derivative},  which immediately implies that property \eqref{eq:Inclusion_opposite_cone} holds. 
\end{step}

\begin{step} 
Without loss of generality, we can assume that $e_{x_0}=e^n$, where $e^n$ denotes the unit vector in $\RR^n$ with all coordinates zero, except for the $n$-th coordinate. Properties \eqref{eq:Inclusion_cone} and \eqref{eq:Inclusion_opposite_cone} can be written in the form:
\begin{align*}
x+ \left(\mathcal{C}_{\eps}(e_x)\cap B_{r_{\eps}/2}'\right)& \subseteq \{v>0\},\\
x- \left(\mathcal{C}_{\eps} (e_x)\cap B_{r_{\eps/2}}'\right)& \subseteq \{v=0\},
\end{align*}
for all $x\in B'_{\eta/2}(x_0)\cap \Gamma(u)$. Taking $x$ sufficiently close to $x_0$, Lemma~\ref{lem:Characteristics_blowup_limits} guarantees that 
$$
\mathcal{C}_{\eps}(e_x)\cap  B_{r_{\eps}/2}'\supset \mathcal{C}_{2\eps}(e^n)\cap B_{r_{\eps}/4}'.
$$
Hence, there exists a positive constant, $\eta_{\eps}$, such that
\begin{align*}
x+ \left(\mathcal{C}_{2\eps} (e^n)\cap B_{r_{\eps}/4}'(x_0)\right)&\subseteq \{v>0\},\\
x- \left(\mathcal{C}_{2\eps} (e^n)\cap B_{r_{\eps}/4}'(x_0) \right) &\subseteq \{v=0\},
\end{align*}
for any $\bar x\in B_{\eta_{\eps}}'(x_0)\cap \Gamma(u)$. Now, fixing $\eps=\eps_0$, by the standard arguments, we can conclude that there exists a
Lipschitz function, $g:\RR^{n-1}\to \RR$, with $|\nabla g|\leq C_n/\eps_0$, such that
\begin{align*}
B_{\eta_{\eps_0}}'(x_0)\cap\{v(\cdot,0)=0\} &= B_{\eta_{\eps_0}}'(x_0)\cap \{x_n\leq g(x')\},\\
B_{\eta_{\eps_0}}'(x_0)\cap\{v(\cdot,0)>0\} &= B_{\eta_{\eps_0}}'(x_0)\cap \{x_n>g(x')\}.
\end{align*}
\end{step}

\begin{step}
Using the normalization $e_{x_0}=e^n$, and letting $\eps$ tend to 0, we see that $\Gamma(u)$ is differentiable at $x_0$ with normal
$e_{x_0}$. Recentering at any $x\in B'_{\eta_{\eps_0}}(x_0)\cap \Gamma(u)$, we see that $\Gamma(u)$ has a normal $e_x$ at $x$. Finally, noting that by Lemma~\ref{lem:Characteristics_blowup_limits} the mapping $x\mapsto e_x$ is $C^{\gamma}$, we obtain that the function $g$ belongs to $C^{1+\gamma}$.
\end{step}

The proof of Theorem~\ref{thm:Regularity_free_boundary} is now complete.
\end{proof}

We conclude \S\ref{sec:Regularity_free_boundary} with the 

\begin{proof}[Proof of Theorem~\ref{thm:Regularity_free_boundary_with_drift}]
It follows immediately from Theorem~\ref{thm:Regularity_free_boundary}, and the reduction procedure described in \S\ref{sec:No_drift}.
\end{proof}

\appendix

\section{Auxiliary results}
\label{sec:Auxiliary_results}

In this section we collect various results that we use in the proofs in the main body of our article.
We first prove an upper bound on the height function $v_{x_0}$ defined in \eqref{eq:Height_function} which we use in the proof of Lemma~\ref{lem:Difference_homogeneous_res}.

\begin{lem}[Growth of the function $v_{x_0}$ on $B_r$]
\label{lem:Growth_v_balls}
Let $v_{x_0}$ be the height function defined in \eqref{eq:Height_function}, where $u\in C^{1+s}(\RR^n)$ is a solution to problem \eqref{eq:Obstacle_problem_without_drift}, with obstacle function $\varphi\in C^{1+s}(\RR^n)$. Then, there exists $C=C(n,s,\|u\|_{C^{1+s}(\RR^n)}, \|\varphi\|_{C^{1+s}(\RR^n)}) > 0$ such that for all $r\in (0,1)$ and every $x_0\in \Gamma(u)$, one has
\begin{equation}
\label{eq:Growth_v_balls}
\|v_{x_0}(x_0+\cdot,\cdot)\|_{C(\bar B_r)} \leq Cr^{1+s}.
\end{equation}
\end{lem}

\begin{proof}
Without loss of generality, we can assume that $x_0=0$. We denote $w(x):=u(x)-\varphi(x)$, where $u$ is a $C^{1+s}(\RR^n)$ solution to the obstacle problem \eqref{eq:Obstacle_problem_without_drift}. Because the functions $u$ and $\varphi$ belong to $C^{1+s}(\RR^n)$, we have that $w\in C^{1+s}(\RR^n)$, and 
$$
w(0)=0,\quad\text{and}\quad \nabla_x w(0)=0.
$$
From definition \eqref{eq:Fractional_laplacian} of the fractional Laplacian operator, property \eqref{eq:Dirichlet_to_Neumann_map}, the fact that $u$ solves \eqref{eq:Obstacle_problem_without_drift} and $0\in\Gamma(u)$, we also have that
$$
\lim_{y\downarrow 0} |y|^av_y(0,y)=0.
$$
Since $u(x,y)$ and $\varphi(x,y)$ are the $L_a$-harmonic extensions of the functions $u(x)$ and $\varphi(x)$ from $\RR^n$ to $\HS$, we have 
\begin{equation}
\label{eq:Definition_u_varphi_Poisson}
\psi(x,y):=\int_{\RR^n} P(z,y) \psi(x-z)\, dz,\quad (x,y)\in\HS,
\end{equation}
where $\psi=u$ or $\psi=\varphi$, and  
$P$ denotes the Poisson kernel 
\begin{equation}
\label{eq:Poisson_kernel}
P(x,y)= C_{n,s} \frac{y^{2s}}{\left(|x|^2+y^2\right)^{(n+2s)/2}},\quad
(x,y)\in \HS,
\end{equation}
for an appropriate $C_{n,s}>0$. Because $u$ solves problem \eqref{eq:Obstacle_problem_without_drift} and $0\in\Gamma(u)$, we have that $(-\Delta)^su(0)=0$. Combining this fact with equalities \eqref{eq:Fractional_laplacian} and \eqref{eq:Definition_u_varphi_Poisson}, we see from \eqref{eq:Height_function} that we can write $v$ in the form
\begin{equation}
\label{eq:Definition_v_Poisson}
v(x,y):= C_{n,s} \int_{\RR^n} \frac{y^{2s}}{\left(|z|^2+y^2\right)^{(n+2s)/2}} w(x-z)\, dz - C_{n,s} \int_{\RR^n} \frac{y^{2s}}{|z|^{n+2s}} w(z)\, dz.
\end{equation}
We next want to show that there is a positive constant, $C=C(\|u\|_{C^{1+s}(\RR^n)}, \|\varphi\|_{C^{1+s}(\RR^n)})$, such that
\begin{align}
\label{eq:Growth_x_direction}
|v(x,y)-v(0,y)| &\leq C|x|^{1+s},\quad\forall (x,y), (0,y) \in B_1,\\
\label{eq:Growth_y_direction}
|v(0,y)| &\leq C|y|^{1+s},\quad\forall (0,y)\in B_1.
\end{align}
It is clear that if we establish \eqref{eq:Growth_x_direction} and \eqref{eq:Growth_y_direction} the proof of the lemma will be concluded since  \eqref{eq:Growth_v_balls} follows immediately from them. Inequality \eqref{eq:Growth_y_direction} can be proved in exactly the same way as \cite[Inequality (2.107)]{Petrosyan_Pop}, with the observation that in its proof we replace the functions $\psi(x,y)$ and $\psi_0(|z|)-\psi_0(0)$ with $v(x,y)$ and $w(z)$, respectively. It only remains to discuss inequality \eqref{eq:Growth_x_direction}. Using the representation formula \eqref{eq:Definition_v_Poisson}, we have that
$$
|v(x,y)-v(0,y)-\nabla_x v(0,y)\dotprod x| \leq \int_{\RR^n} P(z,y)|w(x-z)-w(-z)-\nabla_z w(-z,y)\dotprod x|\, dz,
$$
and using the fact that $w$ belongs to $C^{1+s}(\RR^n)$, and $P(\cdot,y)$ is a probability density, it follows that
\begin{equation}
\label{eq:Growth_x_direction_1}
|v(x,y)-v(0,y)-\nabla_x v(0,y)\dotprod x| \leq C|x|^{1+s},
\end{equation}
where $C=C(\|u\|_{C^{1+s}(\RR^n)}, \|\varphi\|_{C^{1+s}(\RR^n)})$ is a positive constant. Because we have $\nabla_x w(0) =0$, it follows that
\begin{align*}
|\nabla_x v(0,y)| &\leq \int_{\RR^n} P(z,y) |\nabla_z w(z) - \nabla_z w(0)|\, dz \\
&\leq C_{n,s}\int_0^{\infty}\int_{\partial B'_1} \frac{1}{(1+t^2)^{(n+2s)/2}} |\nabla_z w(ty\omega) - \nabla_z w(0)|\,d\sigma(\omega)\, dt\quad\text{(writing $z=t\omega$)}\\
&\leq C |y|^s \int_0^{\infty}\frac{t^s}{(1+t^2)^{(n+2s)/2}}\, dt,
\end{align*}
where in the last inequality we used the fact that $w\in C^{1+s}(\RR^n)$, and $C$ is a positive constant depending on $n$, $s$, $\|u\|_{C^{1+s}(\RR^n)}$, and $\|\varphi\|_{C^{1+s}(\RR^n)}$. We also see that the integral in the last inequality is finite, and so we obtain that
$$
|\nabla_x v(0,y)| \leq C|y|^s.
$$
The preceding inequality together with \eqref{eq:Growth_x_direction_1} yield estimate \eqref{eq:Growth_x_direction}. This concludes the proof of Lemma~\ref{lem:Growth_v_balls}.
\end{proof}

In the proof of Lemma~\ref{lem:Asymptotic_expansion_homogeneous_sol} below we make use of the following result.

\begin{lem}[Regularity in the $x'$-variables]
\label{lem:Regularity_x_prim}
Let $s\in (0,1)$, and $u \in H^1(B_1, |y|^a)$ be a weak solution to equation \eqref{eq:Equation_with_partial_DirichletBC}. Then, for all $r\in (0, 1)$ and all $\alpha \in \NN^{n-1}$, we have that 
$$
D^{\alpha}_{x'} u \in H^1(B_r, |y|^a) \cap L^{\infty}(B_r)
$$
and the derivative $D^{\alpha}_{x'} u$ is a weak solution to equation \eqref{eq:Equation_with_partial_DirichletBC} on $B_r$. Moreover, there exists $C=C(\alpha, n, r, s)>0$ such that
\begin{equation}
\label{eq:Regularity_x_prim}
\|D^{\alpha}_{x'} u\|_{H^1(B_r, |y|^a)} + \|D^{\alpha}_{x'} u\|_{L^{\infty}(B_r)} \leq C \|u\|_{H^1(B_1, |y|^a)}.
\end{equation}
\end{lem}

\begin{proof}
By definition, because $u \in H^1(B_1, |y|^a)$ is a weak solution to \eqref{eq:Equation_with_partial_DirichletBC}, it follows that for all test functions $\varphi \in C^{\infty}_0(B_1\setminus\{x<0, y=0\})$ one has
\begin{equation}
\label{eq:Weak_solution}
\int_{B_1} \nabla u \dotprod \nabla \varphi |y|^a = 0.
\end{equation}
Denoting by $H^1_0(B_1\setminus \{x_n\leq 0, y=0\})$ the closure of
$C^{\infty}_0(B_1\setminus\{x_n\leq 0, y=0\})$ with respect to the
$H^1(B_1, |y|^a)$-norm, the preceding equality holds for all test
functions $\varphi$ that belong to $H^1_0(B_1\setminus \{x_n\leq 0, y=0\})$.

Let $r\in (0,1)$, $h \in (0, (1-r)/4)$, and $e_i \in \RR^{n-1}$, with $i = 1, 2, \ldots, n-1$, be the unit vector in the standard Euclidean basis. We first prove the statement of the Lemma~\ref{lem:Regularity_x_prim} for $\alpha = e_i$, and then an induction argument can easily be applied to obtain the conclusion for all $\alpha \in \NN^{n-1}$. Consider the finite difference operator
$$
D^i_h u(x', x_n, y) = \frac{u(x'+he_i,x_n,y) - u(x',x_n,y)}{h},\quad\forall\, (x',x_n,y)\in B_{1-h}. 
$$
Choosing $\varphi = \eta D^i_h u$ with $\eta\in C^{\infty}_0(B_{1-2h})$, we see that $\varphi \in H^1_0(B_1\setminus\{x_n<0, y=0\}, |y|^a)$, and identity \eqref{eq:Weak_solution} gives 
$$
\int_{B_1} \left|\nabla D^i_h u\right|^2 \eta^2 |y|^a= -2 \int_{B_1} \nabla D^i_hu \dotprod \nabla \eta D^i_h u \eta |y|^a,
$$
from which it follows that
$$
\int_{B_1} \left|\nabla D^i_h u\right|^2 \eta^2 |y|^a \leq 4 \int_{B_1} |D^i_hu|^2 |\nabla \eta|^2 |y|^a.
$$
Choosing $\eta\in C^{\infty}_0(B_1)$ such that
$$
\eta \equiv 1 \text{ on } B_r \quad \text{and} \quad \eta \equiv 0 \text{ on } B^c_{(1+r)/2},
$$
the preceding inequality implies the existence of $C=C(n,r,s)>0$ such that
$$
\int_{B_r} \left|\nabla D^i_h u\right|^2 |y|^a \leq C \int_{B_{(1+r)/2}} |D^i_hu|^2 |y|^a.
$$
An immediate generalization of \cite[Theorem~5.8.3 (i)]{Evans} to our weighted Sobolev spaces gives
\begin{equation}
\label{eq:Uniform_bound_finite_difference}
\int_{B_{(1+r)/2}} | D^i_hu|^2 |y|^a \leq C \int_{B_1} |\nabla u|^2 |y|^a,
\end{equation}
for a $C>0$ and for all $h\in (0, (1-r)/4)$. Combining the preceding two inequalities with the generalization of \cite[Theorem~5.8.3 (ii)]{Evans} to our weighted Sobolev spaces, it follows that $u_{x_i} \in H^1(B_r, |y|^a)$, and 
\begin{equation}
\label{eq:Higher_H_1_norm_estimate}
\|u_{x_i}\|_{H^1(B_r, |y|^a)} \leq C \|\nabla u\|_{L^2(B_1, |y|^a)},
\end{equation}
where $C=C(n,r,s)>0$.

It is now easy to see that identity \eqref{eq:Weak_solution} holds with $u$ replaced by $D^i_h u$. Using the uniform bound \eqref{eq:Uniform_bound_finite_difference} on the $H^1(B_{(1+r)/2}, |y|^a)$-norm of the finite differences, we can take a weak limit along a subsequence $h_n\rightarrow 0$, to conclude that identity \eqref{eq:Weak_solution} holds with $u$ replaced by $u_{x_i}$. Clearly, the derivative $u_{x_i} = 0$ on $B_r\cap \{x_n<0, y=0\}$ in the trace sense in $H^1(B_r, |y|^a)$, and so we obtain that $u_{x_i}$ is a weak solution in $B_r$ to equation \eqref{eq:Equation_with_partial_DirichletBC}.

Because the domain $B_1\setminus\{x_n<0, y=0\}$ is not required to satisfy an exterior cone condition, we may apply \cite[Lemma~2.4.1]{Fabes_Kenig_Serapioni_1982a} to conclude that there is a positive constant, $C=C(n,r,s)$, such that
\begin{equation}
\label{eq:Higher_L_infty_norm_estimate}
\|u_{x_i}\|_{L^{\infty}(B_r)} \leq C \|u\|_{H^1(B_1, |y|^a)}.
\end{equation}
Combining the norm estimates \eqref{eq:Higher_H_1_norm_estimate} and \eqref{eq:Higher_L_infty_norm_estimate}, we obtain inequality \eqref{eq:Regularity_x_prim} with $\alpha = e_i$, for all $i = 1, 2, \ldots, n-1$. The statement for all $\alpha \in \NN^{n-1}$ follows by an induction argument.
\end{proof}

The following asymptotic expansion of homogeneous solutions to equation \eqref{eq:Equation_with_partial_DirichletBC} around the origin is a crucial tool in the proof of Theorem~\ref{T:epi} above.

\begin{lem}
\label{lem:Asymptotic_expansion_homogeneous_sol}
Let $s\in (0,1)$ and $u \in H^1(B_1, |y|^a)$ be a homogeneous function of degree $1+s$. Assume that $u$ is a weak solution to
\begin{equation}
\label{eq:Equation_with_partial_DirichletBC}
\begin{aligned}
L_a u = 0&\quad\text{on } B_1\setminus\{x_n\leq 0, y = 0\},\\
u=0&\quad\text{on } B_1\cap \{x_n\leq 0, y=0\}.
\end{aligned}
\end{equation}
Then, there exist real constants, $c_0, c_1,\ldots, c_{n-1}$, such that
\begin{equation}
\label{eq:Asymptotic_expansion_homogeneous_sol}
\begin{aligned}
u(x',x_n,y) &= \left(x_n+\sqrt{x_n^2+y^2}\right)^s \left[ c_0\left(x_n-s\sqrt{x_n^2+y^2}\right) + \sum_{i=1}^{n-1} c_i x_i \right].
\end{aligned}
\end{equation} 
\end{lem}

\begin{proof} 
Because the function $u$ is homogeneous of degree $1+s$, the second order derivatives $u_{x_ix_j}$ are homogeneous functions of degree $-1+s$. By Lemma~\ref{lem:Regularity_x_prim}, the derivatives $u_{x_ix_j}$ are also bounded, for all $i,j =1, \ldots, n-1$, and so
\begin{equation}
\label{eq:Second_order_derivatives_0}
u_{x_ix_j} = 0\quad\text{on } B_1,\quad\forall\, i,j = 1,\ldots, n-1.
\end{equation}
On $B_1\setminus\{y=0\}$, the weak solution $u$ is a smooth function because the operator $L_a$ has smooth coefficients and is locally strictly elliptic (therefore, $L_a$ is hypoelliptic). Denoting
\begin{align*}
\cB_{1/2} :=\{(x_n, y)\in \RR^2\mid x_n^2+y^2<1/4\},
\quad\text{and}\quad
\cB^\pm_{1/2} := \cB_{1/2} \cap \{y>(<)0\},
\end{align*}
and defining
$$
a_0(x_n, y) := u(0, x_n, y),\quad\text{and}\quad a_i(x_n, y) := u_{x_i}(0, x_n, y),\quad\forall\, (x_n, y)\in \cB^{\pm}_{1/2},
$$
we can write the function $u$ in the form
\begin{equation}
\label{eq:Decomposition_u}
u(x', x_n, y) = a_0(x_n, y) + \sum_{i=1}^{n-1} a_i(x_n, y) x_i,
\end{equation}
for all $(x_n, y) \in \cB^{\pm}_{1/2}$ and $|x'|<1/2$. By construction, the function $a_0(x_n, y)$ is homogeneous of degree $1+s$, and the functions $a_i(x_n,y)$, for $i=1, \ldots, n-1$, are homogeneous of degree $s$. Because $u$ and $u_{x_i}$ are weak solutions to equation \eqref{eq:Equation_with_partial_DirichletBC} on $B_1$, it follows from \cite[Theorems~2.3.12 and 2.4.6]{Fabes_Kenig_Serapioni_1982a} that they are continuous functions on $B_1\setminus\{x_n=y=0\}$. Thus, the functions $a_i(x_n,y)$ are continuous on $\cB_{1/2}\setminus\{x_n=0\}$. Because they have a positive degree of homogeneity, it follows that the functions $a_i(x_n,y)$ are continuous on $\cB_{1/2}$, for all $i=0,1,\ldots,n-1$. 

For all $i=1,\ldots, n-1$, we have that 
$$
a_i(x_n, y) = u_{x_i}(x',x_n,y),
$$
for all $(x_n, y) \in \cB^{\pm}_{1/2}$ and $|x'|<1/2$, which implies by Lemma~\ref{lem:Regularity_x_prim} that the function $a_i(x_n, y)$ belongs to $H^1(\cB^{\pm}_{1/2}, |y|^a)$, and it is a weak solution to equation \eqref{eq:Equation_with_partial_DirichletBC} on $\cB^{\pm}_r$. Moreover, $a_i(x_n,y)$ is continuous up to $y=0$ and $a_i(x_n, 0) = 0$, when $x_n<0$. Because $a_i(x_n, y)$ is homogeneous of degree $s$, it follows that there is a constant $c_i$ such that $a_i(x_n, 0) = c_i x_n^s$, when $x_n>0$.

Because the functions $u \in H^1(B_1, |y|^a)$ and $a_i \in H^1(\cB^{\pm}_{1/2}, |y|^a)$, for all $i=1,\ldots, n-1$, are continuous weak solutions to equation \eqref{eq:Equation_with_partial_DirichletBC}, it follows from identity \eqref{eq:Decomposition_u} that the function $a_0(x_n, y)$ belongs to $H^1(\cB^{\pm}_{1/2}, |y|^a)$, and is also a continuous weak solution to equation \eqref{eq:Equation_with_partial_DirichletBC}. Similarly to the functions $a_i(x_n, y)$, for $i=1,\ldots, n-1$,  the function $a_0(x_n, y)$ satisfies the boundary condition $a_0(x_n,0) = 0$, when $x_n<0$, and there is a constant $c_0$ such that $a_0(x_n, y) = c_0x_n^{1+s}$, when $x_n>0$.

For all $i=1,\ldots, n-1$, we now show that $a_i(x_n, y)$ can be written is the form
\begin{equation}
\label{eq:Form_a_i_nonzero}
a_i(x_n, y ) = \frac{c_i}{2^s}\left(x_n+\sqrt{x_n^2+y^2}\right)^s.
\end{equation}
In polar coordinates, we can write the function in the form $a_i(x_n, y) = b_i(r, \theta) = r^{s}\varphi_i(\theta)$. Because $L_a a_i =0$ on $\cB^{\pm}_{1/2}$, we obtain that the function $\varphi_i(\theta)$ satisfies the second order ordinary differential equation
$$
\sin \theta\,\varphi_{\theta\theta} +a\cos\theta\,\varphi_{\theta} + (as+(1+s)^2)\sin\theta\, \varphi = 0\quad\text{on } (0,\theta),
$$
with Dirichlet boundary conditions
$$
\varphi(0) = \frac{c_i}{2^s} \quad\text{and}\quad\varphi(\pi) = 0,
$$
and so, it has a unique solution. A direct calculation gives that the function
$$
\varphi(\theta) = \frac{c_i}{2^s}(\cos\theta+1)^s,\quad\forall\, \theta \in [0,1],
$$
satisfies the preceding conditions. Thus, the function $a_i(x_n, y)$ indeed takes the form \eqref{eq:Form_a_i_nonzero}.

A similar argument implies that the function $a_{0}(x_n,y)$ must take the form
\begin{equation}
\label{eq:eq:Form_a_0}
a_0(x_n, y ) = \frac{c_0}{2^s(-1+s)}\left(x_n+\sqrt{x_n^2+y^2}\right)^s\left(x_n-s\sqrt{x_n^2+y^2}\right).
\end{equation}
Identities \eqref{eq:eq:Form_a_0}, \eqref{eq:Form_a_i_nonzero} and \eqref{eq:Decomposition_u} give us the precise form \eqref{eq:Asymptotic_expansion_homogeneous_sol} of the function $u(x)$. This concludes the proof.
\end{proof}
%
%

\bibliography{mfpde}

\def\cprime{$'$} \def\polhk#1{\setbox0=\hbox{#1}{\ooalign{\hidewidth
  \lower1.5ex\hbox{`}\hidewidth\crcr\unhbox0}}} \def\cprime{$'$}
  \def\cprime{$'$} \def\cprime{$'$}
  \def\lfhook#1{\setbox0=\hbox{#1}{\ooalign{\hidewidth
  \lower1.5ex\hbox{'}\hidewidth\crcr\unhbox0}}} \def\cprime{$'$}
  \def\cprime{$'$} \def\cprime{$'$} \def\cprime{$'$} \def\cprime{$'$}
\providecommand{\bysame}{\leavevmode\hbox to3em{\hrulefill}\thinspace}
\providecommand{\MR}{\relax\ifhmode\unskip\space\fi MR }
\providecommand{\MRhref}[2]{%
  \href{http://www.ams.org/mathscinet-getitem?mr=#1}{#2}
}
\providecommand{\href}[2]{#2}
\begin{thebibliography}{10}

\bibitem{Applebaum}
D.~Applebaum, \emph{L\'evy processes and stochastic calculus}, second ed.,
  Cambridge Studies in Advanced Mathematics, vol. 116, Cambridge University
  Press, Cambridge, 2009.

\bibitem{Caffarelli_Salsa_Silvestre_2008}
L.~A. Caffarelli, S.~Salsa, and L.~Silvestre, \emph{Regularity estimates for
  the solution and the free boundary of the obstacle problem for the fractional
  {L}aplacian}, Invent. Math. \textbf{171} (2008), 425--461.

\bibitem{Caffarelli_Silvestre_2007}
L.~A. Caffarelli and L.~Silvestre, \emph{An extension problem related to the
  fractional {L}aplacian}, Comm. Partial Differential Equations \textbf{32}
  (2007), no.~7-9, 1245--1260.

\bibitem{Daskalopoulos_Feehan_statvarineqheston}
P.~Daskalopoulos and P.~M.~N. Feehan, \emph{Existence, uniqueness, and global
  regularity for variational inequalities and obstacle problems for degenerate
  elliptic partial differential operators in mathematical finance},
  arXiv:1109.1075.

\bibitem{Evans}
L.~C. Evans, \emph{Partial differential equations}, American Mathematical
  Society, Providence, RI, 1998.

\bibitem{Fabes_Kenig_Serapioni_1982a}
E.~B. Fabes, C.~E. Kenig, and R.~P. Serapioni, \emph{The local regularity of
  solutions of degenerate elliptic equations}, Comm. Partial Differential
  Equations \textbf{7} (1982), 77--116.

\bibitem{Garofalo_Petrosyan_SmitVegaGarcia}
N.~Garofalo, A.~Petrosyan, and M.~Smit Vega~Garcia, \emph{An epiperimetric
  inequality approach to the regularity of the free boundary in the {S}ignorini
  problem with variable coefficients}, arXiv:1501.06498.

\bibitem{Molchanov_Ostrovskii_1969}
S.~A. Molchanov and E.~Ostrovskii, \emph{Symmetric stable processes as traces
  of degenerate diffusion processes}, Theory Probab. Appl. \textbf{14 (1)}
  (1969), 128--131.

\bibitem{Nekvinda_1993}
A.~Nekvinda, \emph{Characterization of traces of the weighted {S}obolev space
  {$W^{1,p}(\Omega,d^\epsilon_M)$} on {$M$}}, Czechoslovak Math. J.
  \textbf{43(118)} (1993), no.~4, 695--711.

\bibitem{Petrosyan_Pop}
A.~Petrosyan and C.~A. Pop, \emph{Optimal regularity of solutions to the
  obstacle problem for the fractional {L}aplacian with drift}, J. Funct. Anal.
  \textbf{268} (2015), no.~2, 417--472.

\bibitem{Silvestre_2007}
L.~Silvestre, \emph{Regularity of the obstacle problem for a fractional power
  of the {L}aplace operator}, Comm. Pure Appl. Math. \textbf{60} (2007), no.~1,
  67--112.

\bibitem{Weiss_1998}
G.~S. Weiss, \emph{Partial regularity for weak solutions of an elliptic free
  boundary problem}, Communications in Partial Differential Equations
  \textbf{23:3-4} (1998), 439--455.

\bibitem{Weiss_1999}
G.~S. Weiss, \emph{A homogeneity improvement approach to the obstacle problem},
  Invent. Math. \textbf{138} (1999), 23--50.

\end{thebibliography}
\bibliographystyle{amsplain}

\end{document}